\documentclass[12 pt]{amsart}
\usepackage{amsthm}
\usepackage{hyperref}
\usepackage{a4wide}
\usepackage{array, xcolor}
\usepackage{longtable}
\usepackage{amsfonts}
\usepackage{amsmath}
\usepackage{mathtools}
\usepackage{accents}
\usepackage{amssymb}
\usepackage{graphicx}
\usepackage{verbatim}
\usepackage{times}
\newtheorem{proposition}{Proposition}
\newtheorem{lemma}[proposition]{Lemma}

\newtheorem{theorem}[proposition]{Theorem}

\theoremstyle{definition}
\newtheorem{definition}[proposition]{Definition}
\newtheorem{remark}[proposition]{Remark}
\newtheorem*{remark*}{Remark}
\newtheorem*{example*}{Example}

\newcommand{\HSp}{\operatorname{M_{heat}}}
\newcommand{\Dd}{\operatorname{D}}
\newcommand{\ff}{\operatorname{ff}}
\newcommand{\lb}{\operatorname{lb}}

\newcommand{\lp}{\left(}
\newcommand{\rp}{\right)}
\newcommand{\EE}{\mathcal{E}}
\newcommand{\PP}{\operatorname{P}}
\newcommand{\Diag}{\operatorname{Diag}}
\newcommand{\FP}{\operatorname{FP}}
\newcommand{\Res}{\operatorname{Res}}
\newcommand{\DD}{\operatorname{\Delta}}
\newcommand{\Ker}{\operatorname{Ker}}
\newcommand{\jj}{\operatorname{j}}
\newcommand{\xx}{\operatorname{x}}
\newcommand{\dvol}{\operatorname{dvol}}
\newcommand{\Spec}{\operatorname{Spec}}
\newcommand{\End}{\operatorname{End}}
\newcommand{\supp}{\operatorname{supp}}
\newcommand{\im}{\operatorname{Im}}
\newcommand{\id}{\operatorname{id}}

\newcommand{\Tr}{\operatorname{Tr}}

\begin{document}
\title[Heat asymptotic of powers of Laplacians]{Heat kernel asymptotics for real powers of Laplacians}
\author{Cipriana Anghel} \email{cianghel@imar.ro}
\thanks{Institute of Mathematics of the Romanian Academy, Bucharest, Romania}
\address{Institute of Mathematics of the Romanian Academy\\
Calea Grivi\c tei 21\\
010702 Bucharest\\ 
Romania}
\date{\today}
\begin{abstract}
We describe the small-time heat kernel asymptotics of real powers $\DD^r$, $r \in (0,1)$ of a non-negative self-adjoint generalized Laplacian $\DD$ acting on the sections of a hermitian vector bundle $\EE$ over a closed oriented manifold $M$. First we treat separately the asymptotic on the diagonal of $M \times M$ and in a compact set away from it. Logarithmic terms appear only if $n$ is odd and $r$ is rational with even denominator. We prove the non-triviality of the coefficients appearing in the diagonal asymptotics, and also the non-locality of some of the coefficients. In the special case $r=1/2$, we give a simultaneous formula by proving that the heat kernel of $\DD^{1/2}$ is a polyhomogeneous conormal section in $\EE \boxtimes \EE^* $ on the standard blow-up space $\HSp$ of the diagonal at time $t=0$ inside $[0,\infty)\times M \times M$.    
\end{abstract}

\thanks{Acknowledgements: I am grateful to my advisor Sergiu Moroianu for helpful discussions and for a careful reading of the paper. This work was partially supported from the project PN-III-P4-ID-PCE-2020-0794 funded by UEFSCDI}

\maketitle

\section{Introduction} 
Let $\DD$ be a self-adjoint generalized Laplacian acting on the sections of a hermitian vector bundle $\EE$ over an oriented, compact Riemannian manifold $M$ of dimension $n$. Denote by $p_t$ the heat kernel of $\DD$, i.e., the Schwartz kernel of the operator $e^{-t\DD}$. It is known since Minakshisundaram-Pleijel \cite{pleijel} that $p_t(x,y)$ has an asymptotic expansion as $t\searrow 0$ near the diagonal: 
\begin{equation}\label{intro}
p_t(x,y) \stackrel{t \searrow 0}{\sim}  t^{-n/2} e^{-\frac{d(x,y)^2}{4t}} \sum_{j=0}^{\infty} t^j \Psi_j (x,y),
\end{equation}
where $d(x,y)$ is the geodesic distance between $x$ and $y$, and the $\Psi_j$'s  are recursively defined as solutions of certain ODE's along geodesics (see e.g. \cite{berline}, \cite{gauduchon}). This asymptotic expansion applied to $\Dd^*\Dd$, where $\Dd$ is a twisted Dirac operator, plays a leading role in the heat kernel  proofs of the Atiyah-Singer index theorem (see \cite{berver}, \cite{bismut}, \cite{getzler}).

B\"{a}r and Moroianu \cite{art} studied the short-time asymptotic behavior of the heat kernel of $ \DD^{1/m}$, $m \in \mathbb N^*$, for a strictly positive self-adjoint generalized Laplacian $\DD$. They give explicit asymptotic formul\ae \ separately in the case when $t \searrow 0$ along the diagonal $\Diag \subset M \times M$, and when $t$ goes to $0$ in a compact set away from the diagonal. The asymptotic behavior depends on the parity of the dimension $n$ and of the root $m$. More precisely, logarithmic terms appear when $n$ is odd and $m$ is even. They use the Legendre duplication formula, and the more general Gauss multiplication formula for the $\Gamma$ function (see e.g. \cite{mellin}).  Another crucial argument in \cite{art} is to use integration by parts in order to show that the Schwartz kernel $q_{-s}$ of the pseudodifferential operator $\DD^{-s}$, $s \in \mathbb C$, defines a meromorphic function when restricted to the diagonal in $M \times M$. 

\subsection*{Small time heat asymptotic for real powers of $\DD$}
The purpose of this paper is to study the short-time asymptotic of the Schwartz kernel $h_t$ of the operator $e^{-t\DD^r}$, where $r \in (0,1)$ and $\DD$ is a non-negative self-adjoint  generalized Laplacian, like for instance $\DD=\Dd^*\Dd$ for a Dirac operator $\Dd$. We give separate formul\ae \ as $t$ goes to $0$ in $[0, \infty) \times \Diag$, and when $t \searrow 0$ in $[0,\infty) \times K$, where $K \subset M \times M$ is a compact set disjoint from the diagonal. In theorem \ref{asafaradiag}, we obtain that ${h_t}_{\vert [0,\infty) \times K} \in t \cdot \mathcal C^{\infty} \lp [0, \infty) \times K \rp$ is a smooth function vanishing at least to order $1$ at $\{ t=0 \}$. The asymptotic along the diagonal depends on the parity of $n$ (like in \cite{art}) and on the rationality of $r$. In theorem \ref{asdiag}, the most interesting case occurs when logarithmic terms appear. This happens only if $n$ is odd, $r=\frac{\alpha}{\beta}$ is rational, and the denominator $\beta$ is even. In that case,   
\begin{equation}\label{prototip}
\begin{aligned}
{h_t}_{\vert_{\Diag}} \stackrel{t \searrow 0}{\sim} {}&\sum_{j=0}^{(n-1)/2} t^{- \frac{n-2j}{2r}} \cdot A_{-\frac{n-2j}{2r}}  +
\sum^{\infty}_{\substack{j=1  \\ 
\mathclap{\alpha \nmid 2j+1}}}\mkern-6mu t^{\frac{2j+1}{2r}} \cdot A_{\frac{2j+1}{2r}} \\ 
{}&+ \sum^{\infty}_{\substack{k=1  \\ 
\mathclap{\beta \nmid  k}}}\mkern-6mu t^k \cdot A_k  + \sum^{\infty}_{\substack{l=1  \\ 
\mathclap{l \ odd}}}\mkern-6mu t^{l \frac{\beta}{2}} \log t \cdot B_{l}. 
\end{aligned}
\end{equation}
Similar expansions are proved in theorem \ref{asdiag} in all the other cases.
\subsection*{Comparison to previous results}
Fahrenwaldt \cite{fah} studied the off-diagonal short-time asymptotics of the heat kernel of $e^{-t f(P)}$, where $f: [0,\infty) \longrightarrow [0,\infty)$ is a smooth function with certain properties, and $P$ is a positive self-adjoint generalised Laplacian. The function $f(x)=x^r$, $r \in (0,1)$ does not satisfy the third condition in \cite[Hypothesis 3.3]{fah}, which seems to be crucial for the arguments and statements in that paper, so the results of \cite{fah} do not seem to apply here.
 
Duistermaat and Guillemin \cite{duis} give the asymptotic expansion of the heat kernel of $e^{-tP}$, where $P$ is a scalar positive elliptic self-adjoint pseudodifferential operator. The order of $P$ in \cite{duis} seems to be a positive integer. It is claimed in \cite{agrano} that this asymptotic holds true in the context of fiber bundles. Furthermore, Grubb \cite[Theorem 4.2.2]{gru} studied the heat asymptotics for $e^{-tP}$ in the context of fiber bundles when the order of $P$ is positive, not necessary an integer. In theorem \ref{asdiag}, we obtain the vanishing of some terms appearing in \cite[Corollary 4.2.7]{gru} in our particular case when $P=\DD^r$ is a real power of a self-adjoint non-negative generalized Laplacian $\DD$, $r \in (0,1)$. We also show that the remaining terms do not vanish in general.

\begin{theorem}\label{nontriv}
For each $r \in (0,1)$, none of the coefficients in the small time asymptotic expansion of $h_t$ appearing in theorem \ref{asdiag} vanishes identically for every generalized Laplacian $\DD$.
\end{theorem}

The logarithmic coefficients $B_l$ and the coefficients $A_j$ for $j \notin \mathbb Z$ can be computed in terms of the  heat coefficients for $e^{-t\Delta}$ appearing in \eqref{intro}. It is well-known that the heat coefficients of a generalized Laplacian are locally computable in terms of the curvature of the connection on $\EE$, the Riemannian metric of $M$ and their derivatives (see e.g. \cite{berline}). 
This is no longer the case for the coefficients of positive integer powers of $t$ from theorem \ref{asdiag} as we shall see now.

By applying theorem \ref{asdiag} for $r \in (0,1)$ and a set of geometric data, namely a hermitic vector bundle $\EE$ over an oriented, compact Riemannian manifold $(M,g)$, a metric connection $\nabla$ and an endomorphism $F \in \End \EE$, $F^*=F$, we produce an endomorphism $A_l \lp M,g,\EE,h_{\EE},\nabla,F \rp \in \mathcal C^{\infty} \lp M, \End \EE  \rp$ for each index $l$ appearing in \eqref{prototip}.

\begin{definition}\label{def1}
\begin{itemize}
\item[$i)$]  We say that a function $A$ which associates to any set of geometric data $(M,g,\EE,h_{\EE},\nabla, F)$ a section in $\mathcal C^{\infty}(M,\End \EE)$ is \emph{locally computable} if for any two sets of geometric data $(M,g,\EE,h_{\EE},\nabla, F)$,  $(M',g',\EE',h_{\EE'}, \nabla', F')$ which agree on an open set (i.e., there exist an isometry $\alpha: U \longrightarrow U'$ between two open sets $U \subset M$, $U' \subset M'$, and a metric isomorphism $\beta : \EE_{\vert_U} \longrightarrow \EE'_{\vert_{U'}}$ which preserves the connection and $\beta_x \circ F_x \circ \beta_{\alpha(x)}^{-1}=F'_{\alpha(x)}$), we have  
\[ \beta_x \circ A_x \circ \beta_{\alpha(x)}^{-1} = A_{\alpha(x)},   \]
for any $x \in U$.
\item[$ii)$] A scalar function $a$ defined on the set of all geometric data $(M,g,\EE,h_{\EE},\nabla,F)$ with values in $\mathbb C$ is called \emph{locally computable} if there exists a locally computable function $C$  as in $i)$ above such that $a=\int_M \Tr C \dvol_g$ for any $\lp M,g, \EE, h_{\EE}, \nabla, F  \rp$.
\item[$iii)$] A function $A$ as in $i)$ is called \emph{cohomologically locally computable} if there exists a locally computable function $C$ as in $i)$ such that for any $\lp M,g, \EE, h_{\EE}, \nabla, F  \rp$,
\[ \left[ \Tr A \dvol_g  \right] = \left[ \Tr C \dvol_g  \right] \in H^n_{dR} \lp M \rp  . \]
\end{itemize}
\end{definition}
\begin{remark}\label{rrr}
\begin{itemize}
\item[$i)$] If a function $A$ is locally computable,  then the integral $a:=\int_{M} \Tr A \dvol_g$ is locally computable.
\item[$ii)$] If a function $A$ is cohomologically locally computable,  then $a:=\int_{M} \Tr A \dvol_g$ is locally computable.
\end{itemize}
\end{remark}
\begin{theorem}\label{nonloc}
If $r$ is irrational, the heat coefficients $A_k$ in \eqref{prototip} (more generally in theorem \ref{asdiag}) are not locally computable for integer $k \geq 1$. If $r=\frac{\alpha}{\beta}$ is rational, then $A_k$ are not locally computable for $k \in \mathbb N \setminus \{ l \beta : l \in \mathbb N \}$. All the other coefficients can be written in terms of the heat coefficients of $e^{-t\DD}$, hence they are locally computable.
\end{theorem}
Consider the asymptotic expansion in \cite[Corollary~2.2']{duis} for a scalar \emph{admissible} operator, i.e. an elliptic, self-adjoint, positive pseudodifferential operator $P$ of positive \emph{integer} order $d$:
\[
 e^{-tP} \stackrel{t \searrow 0}{\sim} \sum_{l=0}^{\infty} A_l(P) t^{(l-n)/d} + \sum_{k=1}^{\infty} B_k(P) t^k \log t .  
 \]
Gilkey and Grubb \cite[Theorem 1.4]{gilgrubb} proved that the coefficients $a_l(P)$ for $l \geq 0$ and $b_k(P)$ for $k \geq 1$ from the corresponding small-time heat trace expansion
\begin{equation}\label{tras}
 \Tr e^{-tP} \stackrel{t \searrow 0}{\sim} \sum_{l=0}^{\infty} a_l(P) t^{(l-n)/d} + \sum_{k=1}^{\infty} b_k(P) t^k \log t .  
 \end{equation}
are generically non-zero in the above class of admissible operators.
In theorem \ref{nontriv}, we prove the same type of statement. However, in our case the order of the operator $\Delta^r$ is $2r$, thus it is integer only for $r = 1/2$.  Even in this case, the non-vanishing result in theorem \ref{nontriv} is not a consequence of \cite[Theorem 1.4]{gilgrubb}, since in our case we do not consider the whole class of admissible operators of fixed integer order $d$ in the sense of Gilkey and Grubb \cite{gilgrubb}, but the smaller class of square roots of generalised Laplacians.

Furthermore, in \cite[Theorem~1.7]{gilgrubb} is proved that the coefficients $a_l(P)$ in \eqref{tras} corresponding to $t^{(l-n)/d}$, for $(l-n)/d \in \mathbb N$ are not locally computable. Remark that the meaning of "locally computable" in \cite{gilgrubb} is different from our definition \ref{def1}. More precisely, in the definition of Gilkey and Grubb, a locally computable function $A$ has to be a smooth function in the jets of the homogeneous components of the total symbol of the operator. A locally computable coefficient in the sense of Gilkey Grubb \cite{gilgrubb} is clearly locally computable in the sense of definition \ref{def1} $ii)$. 

For $r=1/2$, B\"{a}r and Moroianu \cite{art} remark that for odd $k=1,3,...$, the coefficients $A_k$ in \eqref{prototip} corresponding to $t^k$ appear to be non-local. In section \ref{nonlocal}, we clarify this remark by proving that they are indeed non-local in the sense of definition \ref{def1} $i)$ (theorem \ref{nonloc}). In fact, we prove that the $A_k$'s are not \emph{cohomologically} local. By remark \ref{rrr} $ii)$, it also follows that the integrals $a_k:=\int_M \Tr A_k \dvol_g$ are not locally computable in the sense of definition \ref{def1} $ii)$. Therefore,  the $a_k$'s for odd $k$ are also not locally computable in the sense of Gilkey and Grubb \cite{gilgrubb}.

For  $d=1$, the non local coefficients in the heat expansion \eqref{tras} in \cite{gilgrubb} are $a_{n+1}, a_{n+2},...$, while in our case corresponding to $r=d/2=1/2$, the non local coefficients are $a_1,a_3,...$. Despite some formal resemblances, it appears therefore that the results of the present paper are quite different from those of \cite{gilgrubb}.

\subsection*{The heat kernel as a conormal section}
Recall that a smooth function $f$ on the interior of a manifold with corners is said to be \emph{polyhomogeneous conormal} if for any boundary hypersurface given by a boundary defining function $\theta$, $f$ has an expansion with terms of the form $\theta^k \log^l \theta $ towards $\{ \theta=0 \}$ (only natural powers $l$ are allowed). In \cite{melrose}, Melrose introduced the heat space $M_H^2$ by performing a parabolic blow-up of the diagonal in $M \times M$ at time $t=0$. The new space is a manifold with corners with boundary hypersurfaces given by the boundary defining functions $\rho$ and $\omega_0$. Then the heat kernel $p_t$ has the form $\rho^{-n} \mathcal C^{\infty}(M_H^2)$, and it vanishes rapidly at $\{ \omega_0=0 \}$ (see \cite[Theorem~7.12]{melrose})). 

In the special case $r=1/2$, we are able to give a simultaneous formula for the asymptotic behavior of $h_t$ as $t$ goes to zero \emph{both} on the diagonal and away from it. We can understand better the heat operator $e^{-t \DD^{1/2}}$ on a \emph{linear blow-up} heat space $\HSp$, the usual blow-up of $\{ 0 \} \times \Diag$ in $[0,\infty) \times M \times M$. The new added face is called the \emph{front face} and we denote it $\ff$, while the lift of the old boundary is the \emph{lateral boundary}, denoted $\lb$.
\begin{theorem} \label{simultan}
If $n$ is even,  then the Schwartz kernel $h_t$ of the operator $e^{-t\DD^{1/2}}$ belongs to  $ \rho^{-n}\omega_0 \cdot  \mathcal C^{\infty} (\HSp) $, while if $n$ is odd, $h_t \in \rho^{-n} \omega_0 \cdot \mathcal C^{\infty} (\HSp) + \rho \log \rho \cdot \omega_0 \cdot \mathcal C^{\infty} (\HSp) $.
\end{theorem}
Theorem \ref{simultan} improves the results of \cite{art} twofold: it holds true for non-negative generalized Laplacians, and it describes the small-time asymptotic expansion of $h_t$ simultaneously on the diagonal and away from it by showing that $h_t$ is a polyhomogeneous conormal section on $\HSp$ with values in $\EE \boxtimes \EE^*$.

Note that throughout the paper, integral kernels act on sections by integration with respect to the fixed Riemannian density from $M$ in the second variable, so $h_t$ does not contain a density factor. We feel that in the present context this exhibits more clearly the asymptotic behavior.

Based on the study of the case $r=1/2$ and on the separate asymptotic expansions of the heat kernel $h_t$ of $\DD^r$, $r \in (0,1)$ as $t$ goes to $0$ given in theorems \ref{asafaradiag} and \ref{asdiag}, we can conjecture that the heat kernel $h_t$ is a polyhomogeneous conormal function for \emph{all} $r \in (0,1)$ on a ``\emph{transcendental}" heat blow-up space $M^r_{heat}$ depending on $r$. We leave this as a future project. 

\section{The heat kernel of a generalized Laplacian}
Let $\EE$ be a hermitian vector bundle over a compact Riemannian manifold $M$ of dimension $n$. Consider $\DD$ to be a generalized Laplacian, i.e., a second order differential operator which satisfies
\[ \sigma_2(\DD)(x,\xi)=|\xi|^2 \cdot \id_{\EE}. \]  
For example, if $\nabla$ is a connection on $\EE$ and $F \in \Gamma(\End \EE)$, $F^*=F$, then $\nabla^*\nabla+F$ is a symmetric generalized Laplacian on $\EE$.

Suppose that $\DD$ is self-adjoint. Since $M$ is compact, the spectrum of $\DD$ is discrete and $L^2(M,\EE)$ splits as an orthogonal Hilbert direct sum: 
\[ L^2(M,\EE)=\bigoplus_{\lambda \in \Spec \DD}^{\perp} E_{\lambda},  \]
where $E_{\lambda}$ is the eigenspace corresponding to the eigenvalue $\lambda$ of $\DD$. Moreover, $\dim E_{\lambda} < \infty$ and by elliptic regularity, the eigensections are smooth (see e.g. \cite{spinorial}). Let $e^{-t\DD}$ be the \emph{heat operator} defined  as: 
\[ e^{-t\DD}\Phi=e^{-t\lambda} \Phi, \]
for any $\Phi \in E_{\lambda}$, $\lambda \in \Spec \DD$. 
\begin{definition}
The \emph{heat kernel} of a self-adjoint elliptic pseudo-differential operator $P$ acting on the sections of $\EE$ is the Schwartz kernel of the operator $e^{-tP}$.
\end{definition}
If we denote by $\lbrace \Phi_j \rbrace$ an orthonormal Hilbert basis of $\DD$-eigensections, then the heat kernel $p_t(x,y)$ satisfies:
\[ p_t(x,y)=\sum_{j}e^{-t\lambda_j} \Phi_j(x) \otimes \Phi_j^*(y), \] 
in $\mathcal C^{\infty} \lp (0,\infty) \times M \times M \rp$.

Recall that the $L^2$-product of two sections $s_1, s_2 \in \Gamma(\EE)$ is  given by \[ \langle s_1,s_2 \rangle_{L^2(\EE)}=\int_M h_{\EE}(s_1,s_2) \dvol_g ,\] where $g$ is the metric on $M$ and $h_{\EE}$ is the hermitian product on $\EE$. 

Let $y \in M$ be a fixed point. We work in geodesic normal coordinates defined by the exponential map 
\[ \exp_y: T_yM\longrightarrow M. \]
Since $M$ is compact, there exists a global injectivity radius $\epsilon$. For $x$ close enough to $y$ ($d(x,y) \leq \epsilon$), take $\xx \in T_yM$ the unique tangent vector of length smaller than $\epsilon$ such that $x=\exp_y\xx$. Let
\[ \jj(\xx)=\frac{\exp_y^* dx}{d\xx}, \]
namely the pull-back of the volume form $dx$ on $M$ through the exponential map $\exp_{y}$ is equal with $\jj(\xx)d\xx$. More precisely,
\[ \jj(\xx)=\vert \det \lp d_{\xx}\exp_{x_0} \rp \vert={\det }^{1/2} \lp g_{ij}(\xx) \rp .\]
Denote by $\tau_{x}^{y}: \EE_x \longrightarrow \EE_y$ the parallel transport along the unique minimal geodesic $x_s=\exp_{y} (s\xx)$, where $s \in [0,1]$, which connects the points $x$ and $y$. The heat kernel $p_t(x,y)$ belongs to the space $ \mathcal C^{\infty} \lp (0,\infty) \times M \times M, \EE_x \otimes \EE_y^* \rp$ and $p_t(x,y)$ satisfies the heat equation 
\[ \lp \partial_t+{\DD}_x \rp p_t(x,y)=0.\]
Furthermore, $ \lim_{t\rightarrow 0} P_t s=s,$ in $\Vert \cdot \Vert_0$, for any smooth section $s \in \Gamma(M, \EE)$, where 
\[(P_t s)(x)=\int_M p_t(x,y)s(y)dg(y),\]
where $dg(y)$ is the Riemannian density of the metric $g$. The next theorem is due to Minakshisundaram and Pleijel (see for instance \cite{pleijel} and \cite{gauduchon}).

\begin{theorem}\label{fi0}
The heat kernel $p_t$ has the following asymptotic expansion near the diagonal:
\begin{align*}
p_t(x,y) \stackrel{t \searrow 0}{\sim} (4 \pi t)^{-n/2} e^{-\frac{d(x,y)^2}{4t}} \sum_{i=0}^{\infty} t^i \Psi_i (x,y),
\end{align*}
where $\Psi_i: \EE_y \longrightarrow \EE_x $ are $\mathcal C^{\infty}$ sections defined near the diagonal. Moreover, the $\Psi_i$'s are given by the following explicit formul\ae:
\begin{align*}
{}&\Psi_0(x,y)=\jj^{-1/2}(\xx)\tau_{y}^{x},  \\
{}&\tau_{x}^y \Psi_i(x,y)=-\jj^{-1/2}(\xx) \int_{0}^{1}s^{i-1} \jj^{-1/2}(x_s)\tau_{x_s}^{y} \DD_x \Psi_{i-1}(x_s,y)ds.
\end{align*}
\end{theorem}
The asymptotic sum in theorem $\ref{fi0}$ can be understood using truncation and bounds of derivatives as in \cite{berline}. We prefer the interpretation given in \cite{melrose}, where the heat kernel $p_t$ is shown to belong to $\rho^{-n} \mathcal C^{\infty}(M_H^2)$ on the parabolic blow-up space $M_H^2$ and to vanish rapidly at the temporal boundary face $\{ \omega_0 =0 \}$ (see section \ref{unified}).
\begin{example*}\label{example}
Consider the $n$-dimensional product Riemannian manifold $M=\lp S^1\rp^n=\mathbb R^n/(2 \pi \mathbb Z)^n$ with the standard metric $g=d\theta_1^2 \otimes...\otimes d\theta_n^2$. Let $\DD_1$ be the Laplacian on $M$ given by the metric $g$. The eigenvalues of $\DD_1$ are $\{k_1^2+...+k_n^2: k_1,...,k_n \in \mathbb Z  \}$. Let $\varphi_l(\xi)=\frac{1}{\sqrt{2\pi}} e^{il\xi}$ be the standard orthonormal basis of eigenfunctions of each $\DD_{S^1}$. Then for $\theta=(\theta_1,...,\theta_n) \in M$, the heat kernel $p_t$ of $\DD_1$ is the following:
\[ p_t(\theta,\theta)= \sum_{(k_1,...,k_n) \in \mathbb Z^n} e^{-t(k_1^2+...+k_n^2)} \varphi_{k_1}(\theta_1) \overline{\varphi_{k_1}(\theta_1)}... \varphi_{k_n}(\theta_n) \overline{\varphi_{k_n}(\theta_n)} .    \]
Since $\varphi_l(\xi)\overline{\varphi_l(\xi)}=\frac{1}{2\pi}$, for any $\xi \in S^1$, we get:
\[ p_t(\theta,\theta)= \tfrac{1}{(2\pi)^n} \sum_{(k_1,...,k_n) \in \mathbb Z^n} e^{-t(k_1^2+...+k_n^2)}  .    \]
Remark that the Fourier transform of the function $f_t: \mathbb R^n \longrightarrow \mathbb R$, $f_t(x)=e^{-t \vert x \vert^2}$ is given by
\[  \hat{f_t}(\xi)=\tfrac{\pi^{n/2}}{t^{n/2}} e^{-\frac{\vert \xi \vert^2}{4t}}.  \]
Using the multidimensional Poisson formula (see for instance \cite{bal}), we obtain that:
\begin{align*}
p_t(\theta,\theta)=\tfrac{1}{(2 \pi)^n} \sum_{k \in \mathbb Z^n} f_t(k)=\sum_{k \in \mathbb Z^n} \hat{f_t}(2 \pi k)=\tfrac{\pi^{n/2}}{(2 \pi )^n} t^{-n/2} + \tfrac{\pi^{n/2}}{(2 \pi )^n} t^{-n/2} \sum_{k \in \mathbb Z^n \setminus \{ 0 \}} e^{-\frac{\pi^2 \vert k \vert^2}{t}}.
\end{align*}
Since the last sum is of order $\mathcal O \lp e^{-\frac{1}{t}} \rp$ as $t \rightarrow 0$, it follows that the first coefficient in the asymptotic expansion at small-time $t$ of $p_t$ is $\tfrac{\pi^{n/2}}{(2 \pi )^n}$ and all the others vanish.
\end{example*}

From now on, suppose that $\DD$ is non-negative (i.e., $h_{\EE} \lp \DD f, f \rp \geq 0$, for any $f \in \mathcal C^{\infty} (M,\EE)$). For $s \in \mathbb C$, we define the complex powers $\DD^{-s} \in \Psi^{-2s} \lp M, \EE \rp$ of $\DD$ as:
\[
\DD^{-s} \Phi= \left\{
    \begin{array}{ll}
        \lambda^{-s}\Phi & \mbox{ if } \Phi \in E_{\lambda}, \ \lambda \neq 0, \\
        0 & \mbox{ if } \Phi \in \Ker \DD.
    \end{array}
\right.
\]
Remark that $(\DD^s)_{s \in \mathbb C}$ is a holomorphic family of pseudodifferential operators. Let $r \in (0,1)$. We denote by $h_t$ the heat kernel of $\DD^r$, namely the Schwartz kernel of the operator $e^{-t\DD^r}$. We have seen that: 
\begin{equation}\label{asp_t} 
p_t(x,x) \stackrel{t \searrow 0}{\sim} t^{-n/2} \sum_{j=0}^{\infty} t^{j}a_{j}(x,x), 
\end{equation} 
with smooth sections $a_j(x,x) \in \EE_x \otimes \EE_x^*$. 

\section{The link between the heat kernel and complex powers of the Laplacian}
\begin{proposition}[Mellin Formula]\label{mellin}
With the notations above, for $\Re s>0$, we have:   
\[ \DD^{-s}=\frac{1}{\Gamma(s)} \int_{0}^{\infty} t^{s-1} \lp e^{-t\DD}  - \PP_{\Ker \DD} \rp  dt , \]
where $\PP_{\Ker \DD}$ is the orthogonal projection onto the kernel of $\DD$.
\end{proposition}
\begin{proof}
It is straight-forward to check that both sides coincide on eigensections $\Phi \in E_{\lambda}$, $\lambda \in \Spec \DD$. Since $\lbrace \Phi_j \rbrace_{j}$ is a Hilbert basis, the result follows.   
\end{proof}

We will write $\PP_{\Ker \DD}(x,y)$ for the Schwartz kernel $\sum_{k} \varphi_k(x) \otimes \varphi_k^*(y)$, where $\{ \varphi_k \}$ is an orthonormal basis in $\Ker \DD$. Denote by $q_{-s}$ the Schwartz kernel of the operator $\DD^{-s}$. Let us first study the poles and the zeros of $q_{-s}$ away from the diagonal. 
\begin{proposition}\label{poliafaradiag}
Let $K$ be a compact in $M \times M \setminus \Diag$. Then for $(x,y)\in K $, the function $s \longmapsto q_{-s} \in \mathcal C^{\infty} \lp K, \EE \boxtimes \EE^*  \rp  $ is entire. Moreover, $q_{-s}$ vanishes at each negative integer $s$. 
\end{proposition}
\begin{proof}
For $\Re s>0$, let $f_{x,y}(s)= \int_{0}^{\infty} t^{s-1}\lp p_t(x,y)-\PP_{\Ker \DD}(x,y) \rp dt $.
Remark that:
\begin{align*}
 f_{x,y}(s)=\int_{0}^{\infty} t^{s-1}\lp p_t(x,y)-\PP_{\Ker \DD}(x,y) \rp dt ={}& \int_{1}^{\infty} t^{s-1} \lp p_t(x,y)-\PP_{\Ker \DD}(x,y) \rp dt \\ +{}& \int_{0}^{1} t^{s-1}p_t(x,y) dt - \PP_{\Ker \DD}(x,y) \cdot \int_{0}^{1} t^{s-1}dt. 
\end{align*}
Since $p_t(x,y)-\PP_{\Ker \DD}(x,y)$ decays exponentially fast as $t$ goes to $\infty$, the first integral is absolutely convergent in $ C^k$ norms. The heat kernel $p_t$ vanishes with all of its derivatives as $t \searrow 0$ in the compact $K$, thus the second integral is also absolutely convergent. The last integral term is well-defined for $\Re s>0$, and it extends to a meromorphic function on $\mathbb C$ with a simple pole in $s=0$. Therefore $s \mapsto f_{x,y}(s)$ extends to a meromorphic function on $\mathbb C$. By proposition \ref{mellin} and the identity theorem, the equality of meromorphic functions:
\[ \Gamma(s)q_{-s}(x,y)= f_{x,y}(s)  \]
holds for any $s \in \mathbb C$. In particular, we obtain $q_0(x,y)=- \PP_{\Ker \DD}(x,y)$. Furthermore, ${q_{-s}}_{\vert_{K}}$ is an entire function and vanishes in $s=-1,-2,...$.   
\end{proof}

Now we check the behavior of $q_{-s}$ along the diagonal. It is no longer holomorphic there, and the coefficients $a_j(x,x)$ from \eqref{asp_t} appear as residues of $q_{-s}(x,x)$.  
\begin{proposition}\label{poligammaq}
Let $x \in M$. Then $\Gamma(s)q_{-s}(x,x)$ is a meromorphic function on $\mathbb C$ with simple poles in $s \in \lbrace 0  \rbrace \cup \lbrace \frac{n}{2}-j : j \in \mathbb N \rbrace$. The residue of $\Gamma(s)q_{-s}(x,x)$ in $s=\frac{n}{2}-j$, $j \neq \frac{n}{2}$, is $a_j(x,x)$. If $n$ is even, then the residue of $\Gamma(s)q_{-s}(x,x)$ in $s=0$ is $a_{\frac{n}{2}}(x,x)-\PP_{\Ker \DD}(x,x)$. If $n$ is odd, the residue in $s=0$ is $-\PP_{\Ker \DD}(x,x)$ and $q_{-s}(x,x)$ vanishes at $s \in \{ -1,-2,... \}$.  
\end{proposition}
\begin{proof}
Consider the function $f_{x,x}(s)=\int_{0}^{\infty} t^{s-1}\lp p_t(x,x)- \PP_{\Ker \DD}(x,x) \rp dt$ for $\Re s > \frac{n}{2}$. We have:
\begin{align*}
f_{x,x}(s)=\int_{0}^{\infty} t^{s-1}\lp p_t(x,x)-\PP_{\Ker \DD}(x,x) \rp dt ={}& \int_{1}^{\infty} t^{s-1} \lp p_t(x,x)-\PP_{\Ker \DD}(x,x) \rp dt \\ +{}& \int_{0}^{1} t^{s-1}p_t(x,x) dt - \PP_{\Ker \DD}(x,x) \cdot \int_{0}^{1} t^{s-1}dt. 
\end{align*}
The first integral is absolutely convergent, as seen in the proof of proposition \ref{poliafaradiag}. The last integral term is meromorphic with a simple pole at $s=0$ with residue $-\PP_{\Ker \DD}(x,x)$. Let us analyze the behavior of the second term $A_x(s)=\int_{0}^{1} t^{s-1}p_t(x,x)dt$. 

Using \eqref{asp_t}, we have that for $N \geq 0$,
\[t^{n/2}p_t(x,x)=\sum_{j=0}^N t^j a_j(x,x)+R_{N+1}(t,x), \]
where $R_{N+1}$ is of order $\mathcal O (t^{N+1})$ as $t \to 0$. Furthermore, we obtain:
\begin{align*}
A_x(s)=\int_{0}^{1} t^{s-\frac{n}{2}-1} t^{\frac{n}{2}}p_t(x,x)dt={}&\sum_{j=0}^N \int_{0}^{1} t^{s-\frac{n}{2}-1} t^j a_j(x,x)dt + \int_{0}^{1} t^{s-\frac{n}{2}-1} R_{N+1}(t,x)dt \\
={}&\sum_{j=0}^N a_j(x,x) \frac{1}{s-\frac{n}{2}+j} + \int_{0}^{1} t^{s-\frac{n}{2}-1} R_{N+1}(t,x)dt.
\end{align*}
Thus $s \mapsto A_x(s)$ extends to a meromorphic function on $\mathbb C$ with simple poles in $\{ \frac{n}{2}-j : \ j=\overline{0,N+1} \}$.  Using again proposition \ref{mellin} and the identity theorem,  we deduce the equality:
\[ \Gamma(s)q_{-s}(x,x)=f_{x,x}(s),\]
for any $s \in \mathbb C$. It follows that $\Gamma(s) q_{-s}(x,x)$ is meromorphic on $\mathbb C$ with simple poles in $s \in \lbrace 0  \rbrace \cup \lbrace \frac{n}{2}-j : j \in \mathbb N \rbrace$. Moreover, the residue of $\Gamma(s)q_{-s}(x,x)$ in a pole $\frac{n}{2}-j$ is $a_j(x,x)$, and the conclusion follows. 
\end{proof}

For $p\in \mathbb C$ and $\epsilon>0$, let $B_\epsilon(p)$ be the open disk centered in $p$ of radius $\epsilon$. We need the following technical result.
\begin{proposition}\label{umbenzi}
Consider $\alpha < \beta$, and let $\epsilon>0$, $l \in \mathbb N$. 
\begin{itemize}
\item[$\bullet$] If $K$ is a compact set disjoint from the diagonal, then the function $ s \longmapsto \Gamma(s){q_{-s}}_{\vert_K} $ is uniformly bounded in $\{ s \in \mathbb C : \alpha \leq \Re s \leq \beta  \} \setminus B_{\epsilon}(0)$ in the $\mathcal C^{l}$ norm on $K$.
\item[$\bullet$] The function $ s \longmapsto  \Gamma(s){q_{-s}}_{\vert_{\Diag}} $ defined on $\{ s \in \mathbb C: \ \alpha \leq \Re s \leq \beta \}\setminus \bigcup_{j \in \mathbb N \cup\lbrace \frac{n}{2} \rbrace } B_{\epsilon}(\frac{n}{2}-j) \longrightarrow \mathcal C^l \lp \Diag, \EE \otimes \EE^* \rp$ is uniformly bounded.
\end{itemize}
\end{proposition}
\begin{proof}
With the same argument as in the proof of proposition \ref{poliafaradiag}, the restriction of the $\mathcal C^l$ norm on $K$ of the function $s \mapsto f_{x,y}(s) $ is absolutely convergent in $\{ s \in \mathbb C : \alpha \leq \Re s \leq \beta  \} \setminus B_{\epsilon}(0)$, hence it is uniformly bounded. 

As in the proof of proposition \ref{poligammaq},  the $\mathcal C^l$ norm along $\Diag$ of $s \longmapsto f_{x,x}(s)$ converges absolutely in $\{ s \in \mathbb C: \ \alpha \leq \Re s \leq \beta \}\setminus \bigcup_{j \in \mathbb N \cup \lbrace \frac{n}{2} \rbrace } B_{\epsilon}(\frac{n}{2}-j)$, thus the conclusion follows.  
\end{proof}

\section{The behavior of quotients of Gamma functions along vertical lines}
A fundamental result used in \cite{art} is the Legendre duplication formula:
\[ \frac{\Gamma(s)}{\Gamma \lp \frac{s}{2} \rp } = \tfrac{1}{\sqrt{2\pi}} 2^{s- \frac{1}{2}} \Gamma\lp \frac{s+1}{2} \rp, \]
together with the rapid decay of the Gamma function in vertical lines $\Re s = \tau$ (see e.g. \cite{mellin}). These results are replaced in our case by the following estimate.
\begin{proposition}\label{rapgamma}
The function $s \longmapsto \frac{\Gamma(s)}{\Gamma(rs)}$ decreases in vertical lines faster than $\vert s \vert^{-k}$, for any $k \geq 0$, uniformly in each strip $\lbrace s \in \mathbb C : \alpha \leq \Re (s) \leq \beta \rbrace$, for any $\alpha,\beta \in \mathbb R$. 
\end{proposition}
\begin{proof}
For $z \in \mathbb C \setminus \mathbb R_{-}$, recall the Stirling formula (see for instance \cite{watson}):
\[ \log \Gamma(z)=\lp z-\frac{1}{2} \rp \log z -z + \frac{1}{2} \log (2\pi) + \Omega(z), \]
where $\log$ is defined on its principal branch, and $\Omega$ is an analytic function of $z$. For $|\arg z|<\pi$ and $|z| \to \infty$, $\Omega$ can be written as
\[\Omega(z)=\sum_{j=1}^{N-1} \frac{B_{2j}}{2j(2j-1)z^{2j-1}}+R_N(z), \]
where $B_{2j}$ are the Bernoulli numbers $\lp B_2=\frac{1}{6}, \ B_4=-\frac{1}{30}, \ B_6=\frac{1}{42}, \text{etc} \rp$. Moreover, the error term satisfies: 
\[|R_N(z)| \leq \frac{|B_{2N}|}{2N(2N-1)} \cdot \frac{\sec^{2N}(\frac{\arg z}{2})}{|z|^{2N-1}}, \]
thus $R_N(z)$ is of order $\mathcal O \lp |z|^{-2N+1} \rp$ as $|z| \to \infty$ (see for instance \cite[(2.1.6)]{mellin}). For $s \notin (-\infty,0)$, it follows that:
\begin{align*}
\frac{\Gamma(s)}{\Gamma(rs)}=s^{-s(r-1)}e^{s(r-1)}r^{\frac{1}{2}-rs}e^{\Omega(s)-\Omega(rs)}.
\end{align*}
Let $s=a+ib$, $a \in \mathbb R$ fixed. As $|b| \to \infty$, the difference $\vert \Omega(s)-\Omega(rs) \vert \to 0$, thus  $\vert e^{\Omega(s)-\Omega(rs)} \vert \to 1$. Note that $\vert r^{\frac{1}{2}-rs} \vert =  \vert r^{\frac{1}{2}-ra} \vert$ and $\vert e^{(r-1)s} \vert = e^{(r-1)a}$, so these terms are bounded. We show in lemma \ref{ss} that for any $k \geq 0$, $\vert s \vert^{k} \vert s^s \vert $ goes to $0$ as $\Re s =a$ is fixed and $\vert \im s \vert$ tends to $\infty$. It follows that the quotient $\frac{\Gamma(s)}{\Gamma(rs)}$ indeed decreases in vertical lines faster than $\vert s \vert ^{-k}$, for any $k \geq 0$, uniformly in vertical strips.
\end{proof}

\begin{lemma}\label{ss}
Let $k \geq 0$. If $a \in \mathbb R$ is fixed and $\vert b \vert \to \infty$, then $\vert (a+ib)^{k+a+ib} \vert$ tends to zero.
\end{lemma}
\begin{proof}
Let $s=a+ib \notin (-\infty,0)$ and set $\log (a+ib)=x+iy$. Then $x=\log \sqrt{a^2+b^2}$, $y=\arg s \in (-\pi,\pi)$, hence
\[\vert s^{s+k} \vert=\vert e^{(k+a+ib) \log (a+ib)} \vert=  e^{(k+a)x-by}  =
 e^{(k+a) \log \sqrt{a^2+b^2}-b \arg s} . \]
Since $b=\tan \arg s \cdot a$, the exponent is equal to:
\begin{equation}\label{exponent}
(k+a)\log \sqrt{a^2+b^2}-b \arg s = (k+a) \log a +\frac{k+a}{2}  \log \lp 1+\tan^2 \arg s \rp -a \tan \arg s \cdot \arg s.
\end{equation}
If $a>0$,  then $\arg s\nearrow \frac{\pi}{2}$ or $ \arg s \searrow -\frac{\pi}{2}$, and in both cases $t:=\tan \arg s$ tends to $\infty$. The exponent \eqref{exponent} behaves as the function $t \longmapsto \log (1+t^2)-t$, therefore as $t \to \infty$, the exponent goes to $-\infty$ and the statement of the claim follows.  

If $a<0$, then $\arg s \searrow \frac{\pi}{2}$ or $\arg s \nearrow -\frac{\pi}{2}$. In the first case when $\arg s \searrow \frac{\pi}{2}$, it follows that $ t = \tan \arg s \to -\infty$. The exponent \eqref{exponent} behaves as $   \pm \log (1+t^2)+t $, hence the conclusion follows. While if  $\arg s \nearrow -\frac{\pi}{2}$, then $t \to \infty$, and the exponent \eqref{exponent} behaves as $ \pm \log (1+t^2)-t   $, thus the exponent tends again to $-\infty$. Therefore $\vert s^{k+s} \vert$ goes to zero, which ends the proof.
\end{proof}

\section{Link between the complex powers of $\DD$ and the heat kernel of $\DD^r$}
\begin{proposition}[Inverse Mellin Formula]\label{inversemellin}
For $\Re \tau >0$, the operators $e^{-t\DD^r}$ and $\DD^{-s}$ are related by the following formula:
\[e^{-t \DD^r} -\PP_{\Ker \DD}= \frac{1}{2\pi i}\int_{\Re s= \tau} t^{-s} \Gamma(s)  \DD^{-rs}  ds. \]
\end{proposition}
\begin{proof}
The equality holds on each eigensection $\Phi_j$ corresponding to an eigenvalue $\lambda_j \in \Spec \DD$. Since $\lbrace \Phi_j \rbrace_{j}$ is a Hilbert basis, the result follows.   
\end{proof}

Set $\tau > \frac{n}{2r}$. Then the Schwartz kernel $q_{-rs}$ of $\DD^{-rs}$ is continuous and by the inverse Mellin formula, we get an identity which relates the Schwartz kernels $h_t$ and $q_{-rs}$:
\begin{align*}
 h_t(x,y)- \PP_{\Ker \DD}(x,y) ={}& \tfrac{1}{2 \pi i} \int_{\Re s=\tau} t^{-s}\Gamma(s) q_{-rs}(x,y) ds \\
={}&\tfrac{1}{2 \pi i} \int_{\Re s = \tau} t^{-s}\frac{\Gamma(s)}{\Gamma(rs)} \cdot \Gamma(rs)q_{-rs}(x,y) ds .
\end{align*}
Now let $k >0$. By changing $\tau$ to $\tau + \epsilon$ (for a small $\epsilon>0$) if needed, we can assume that $\tau-k \notin \lbrace \frac{n}{2}-j :j \in \mathbb N \rbrace \cup \{ 0 \}$. Using propositions \ref{umbenzi} and \ref{rapgamma}, we can apply the residue formula and move the line of integration to the left:
\begin{equation}\label{nke}
\begin{aligned}
 h_t(x,y) ={}&\tfrac{1}{2 \pi i} \int_{\Re s = \tau -k} t^{-s}\frac{\Gamma(s)}{\Gamma(rs)} \cdot \Gamma(rs)q_{-rs}(x,y) ds  \\
 {}& +\sum_{s \in -\mathbb N \cup \lbrace \frac{n-2j}{2r} : \ j \in \mathbb N \rbrace} \Res_{s} \lp t^{-s} \Gamma(s) q_{-rs}(x,y) \rp + \PP_{\Ker \DD}(x,y).
\end{aligned}
\end{equation}
Notice that $-\mathbb N \cup \lbrace \frac{n-2j}{2r} : \ j \in \mathbb N \rbrace$ is the set of all possible poles of $s \mapsto \Gamma(s) q_{-rs}(x,y)$, but some of them might actually be regular points. We will study the sum \eqref{nke} in detail in theorems \ref{asafaradiag} and \ref{asdiag}.

Let $K$ be a compact set in $M \times M \setminus \Diag$ and $l \in \mathbb N$. Remark that the integral term in \eqref{nke} is of order $\mathcal O \lp t^{k-\tau} \rp$ in $\mathcal C^l (K, \EE \boxtimes \EE^*)$. Indeed,
\begin{align*}
\left\Vert  \int_{\Re s = \tau-k} t^{-s}\Gamma(s){q_{-rs}}_{\vert_K} ds   \right\Vert_l \leq t^{-\tau+k} \cdot  \int_{ s = \tau-k+i u} \left\Vert \frac{\Gamma(s)}{\Gamma(rs)} \cdot \Gamma(rs) {q_{-rs}}_{\vert_K}  \right\Vert_l du,
\end{align*}  
and using again propositions \ref{umbenzi} and \ref{rapgamma}, the claim follows. Furthermore, when $k$ goes to $\infty$, we get:
\begin{equation}\label{ht1}
{h_t}_{\vert_K} \stackrel{t \searrow 0}{\sim} \sum_{\alpha=0}^{\infty} t^{\alpha} \cdot \Res_{s=-\alpha} \lp \Gamma(s) {q_{-rs}}_{\vert_K} \rp + t^0 \cdot {\PP_{\Ker \DD}}_{\vert_K}, 
\end{equation}
The meaning of the asymptotic sign in \eqref{ht1} is that if we set $h_t^N$ to be the right hand side in \eqref{ht1} restricted to $\alpha \leq N$, then the difference $\vert \partial_t^j \lp {h_t}_{\vert_K}- h_t^N \rp  \vert$ is of order $\mathcal O (t^{N+1-j})$ in $\mathcal C^l(K, \EE \boxtimes \EE^*)$,  for any $N,j \in \mathbb N$. 

Remark that using again propositions \ref{umbenzi} and \ref{rapgamma}, the integral term in \eqref{nke} is of order $\mathcal O \lp t^{k-\tau} \rp$ in $\mathcal C^l (\Diag, \EE \otimes \EE^*)$. Therefore when $k$ tends to $\infty$, we obtain
\begin{equation}\label{ht2}
{h_t}_{\vert_{\Diag}} \stackrel{t \searrow 0}{\sim} \sum_{\alpha \in \lp -\mathbb N \rp \cup \lbrace \frac{n-2j}{2r}: j \in \mathbb N \rbrace } t^{-\alpha} \cdot \Res_{s=\alpha} \lp \Gamma(s) {q_{-rs}}_{\vert_{\Diag}} \rp + t^0 \cdot {\PP_{\Ker \DD}}_{\vert_{\Diag}}, 
\end{equation}
in the sense of the following:
\begin{definition}
Consider $l \in \mathbb N$ and let $A,B \subset \mathbb R$. We say that $ {h_t}_{\vert_{\Diag}} \stackrel{t \searrow 0}{\sim} \sum_{\alpha \in A} t^\alpha {c_\alpha} + \sum_{\beta \in B} t^{\beta} \log t \cdot  c_{\beta} $ if for any $k,N \in \mathbb N$, the difference
\[  \partial_t^j \lp {h_t}_{\vert_{\Diag}}- \sum_{\alpha \leq N} t^\alpha {c_\alpha} - \sum_{\beta \leq N}  t^{\beta} \log t \cdot  c_{\beta}  \rp  \]
is of order $\mathcal O (t^{N+1-j} \log t)$ in $\mathcal C^l(\Diag, \EE \otimes \EE^*)$.
\end{definition}

\section{The asymptotic expansion of $h_t$ away from the diagonal}
\begin{theorem}\label{asafaradiag}
The Schwartz kernel $h_t$ of the operator $e^{-t\DD^r}$ is $\mathcal C^{\infty}$ on $[0,\infty)\times \lp M \times M \setminus \Diag \rp$. Furthermore, let $K \subset M \times M \setminus \Diag$ be a compact set. Then the Taylor series of ${h_t}_{\vert_{K}}$ as $t \searrow 0$ is the following:
\[{h_t}_{\vert_K} \stackrel{t \searrow 0}{\sim} \sum_{j=1}^{\infty} t^{j} {q_{rj}}_{\vert_K} \frac{(-1)^j}{j!} .  \]
Moreover, if $r=\frac{\alpha}{\beta}$ is rational with $\alpha,\beta$ coprime, then the coefficient of $t^j$ vanishes for $j \in \beta \mathbb N^*$.
\end{theorem} 
\begin{proof}
Let $j \in \mathbb N$. Using propositions \ref{umbenzi} and \ref{rapgamma}, $(-s)(-s-1)...(-s-j+1) t^{-s-j} \frac{\Gamma(s)}{\Gamma(rs) } \Gamma(rs){q_{-rs}}_{\vert_{K}} $ is $L^1$ integrable on $ \Re s = \tau-k$ in $\mathcal C^l (K, \EE \boxtimes \EE^*)$, for sufficiently large $k$ and for any $l \in \mathbb N$. It follows that $h_t$ is $\mathcal C^{\infty}$ on $(0,\infty)\times \lp M \times M \setminus \Diag \rp$. By proposition \ref{poliafaradiag}, the function $s \mapsto q_{-rs}(x,y)$ is entire for any $(x,y) \in K$. Since $\Res_{s=-j} \Gamma(s)= \frac{(-1)^j}{j!} $, using \eqref{ht1} we get:
\[{h_t}_{\vert_K} \stackrel{t \searrow 0}{\sim} \sum_{j=0}^{\infty} t^{j} {q_{rj}}_{\vert_K} \frac{(-1)^j}{j!} +{\PP_{\Ker \DD}}_{\vert_K}.\]
We obtained in the proof of proposition \ref{poliafaradiag} that ${q_0}_{\vert_K}=-{\PP_{\Ker \DD}}_{\vert_K}$, thus:
\[{h_t}_{\vert_K} \stackrel{t \searrow 0}{\sim} \sum_{j=1}^{\infty} t^{j} {q_{rj}}_{\vert_K} \frac{(-1)^j}{j!},   \]
and therefore $h_{t_{\vert_K}}$ is $\mathcal C^{\infty}$ also at $t=0$, and vanishes at order $1$.  Moreover, using again proposition \ref{poliafaradiag}, if $r=\frac{\alpha}{\beta}$ is rational and $j$ is a non-zero  multiple of $\beta$, then $q{_{rj}}_{\vert_K}  \equiv 0$ and the conclusion follows.   
\end{proof}

\section{The asymptotic expansion of $h_t$ along the diagonal} 
To obtain the coefficients in the asymptotic of $h_t$ along the diagonal as $t\searrow 0$, we need to compute the residues from \eqref{ht2}. Some of them are related to the heat coefficients $a_j$'s of $p_t$ due to proposition \ref{poligammaq}. We will distinguish three cases. If $n$ is even, $\Gamma(s)q_{-rs}(x)$ has simple poles in $\lbrace \frac{n}{2r},\frac{n-2}{2r},...,\frac{2}{2r} \rbrace \cup \lbrace 0,-1,... \rbrace$ and the residues will give rise to real powers of $t$. If n is odd and either $r$ is irrational or $r$ is rational with odd denominator, $\Gamma(s)q_{-rs}(x)$ has simple poles in $\lbrace 0,-1...\rbrace \cup \lbrace \frac{n-2j}{2r} : j=0,1,...   \rbrace$. Otherwise, if $n$ is odd and $r$ is rational with even denominator, then there exist some double poles which give rise to logarithmic terms in the asymptotic  expansion of $h_t$.

\begin{theorem}\label{asdiag}
Let $a_j(x,x)$ be the coefficients in \eqref{asp_t} of the heat kernel $p_t$ of the non-negative self-adjoint generalized Laplacian $\DD$. The asymptotic expansion of the Schwartz kernel $h_t$ of the operator $e^{-t\DD^r}$, $r \in (0,1)$ along the diagonal when $t\searrow 0$ is the following:   
\begin{itemize}
\item[$(1)$] If $n$ is even, then
\[ {h_t}_{\vert_{\Diag}} \stackrel{t \searrow 0}{\sim} \sum_{j=0}^{n/2-1} t^{- \frac{n-2j}{2r}} \cdot A_{-\frac{n-2j}{2r}} + a_{n/2} + \sum_{j=1}^{\infty} t^j A_j.  \]
If $r=\frac{\alpha}{\beta}$ is rational, for $j=l \beta$, $ l \in \mathbb N^*$, we obtain that $q_{rj}(x,x)=(-1)^j \cdot j! \cdot a_{\frac{n}{2}+l \alpha}(x,x)$, and the coefficient of $t^{l \beta}$ can be described more precisely as:
\[ A_{l \beta}=a_{\frac{n}{2}+l \alpha}.   \]
\item[$(2)$] 
If $n$ is odd and either $r \in \mathbb R \setminus \mathbb Q$ or the denominator of $r$ is odd, then
\[ {h_t}_{\vert_{\Diag}} \stackrel{t \searrow 0}{\sim} \sum_{j=0}^{(n-1)/2} t^{- \frac{n-2j}{2r}} \cdot A_{-\frac{n-2j}{2r}} + \sum_{j=1}^{\infty} t^j \cdot A_j + \sum_{j=1}^{\infty} t^{\frac{2j+1}{2r}} \cdot A_{\frac{2j+1}{2r}} .  \]
Moreover, if $r=\frac{\alpha}{\beta}$ is rational and $\beta$ is odd, then $ A_{l \beta} \equiv 0$ for any $l \in \mathbb N^*$. 
\item[$(3)$] If $n$ is odd, $r=\frac{\alpha}{\beta}$ is rational and its denominator $\beta$ is even, then 
\begin{align*}
{h_t}_{\vert_{\Diag}} \stackrel{t \searrow 0}{\sim} {}&\sum_{j=0}^{(n-1)/2} t^{- \frac{n-2j}{2r}} \cdot A_{-\frac{n-2j}{2r}}  +  
\sum^{\infty}_{\substack{j=1  \\ 
\mathclap{\alpha \nmid 2j+1}}}\mkern-6mu t^{\frac{2j+1}{2r}} \cdot A_{\frac{2j+1}{2r}} + 
\sum^{\infty}_{\substack{j=1  \\ 
\mathclap{ \frac{\beta}{2} \nmid j}}}\mkern-6mu  t^j \cdot A_j \\ 
{}& + \sum^{\infty}_{\substack{l=1  \\ 
\mathclap{ l \text{ odd} }}}\mkern-6mu  t^{l \frac{\beta}{2}} \cdot A_{l \frac{\beta}{2}}
+ \sum^{\infty}_{\substack{l=1  \\ 
\mathclap{ l \text{ odd} }}}\mkern-6mu  t^{l \frac{\beta}{2}} \log t \cdot B_{l \frac{\beta}{2}}.
\end{align*}
\end{itemize}
In all these cases, the coefficients are
\begin{align*}
{}&A_{-\frac{n-2j}{2r}}(x)= \frac{\Gamma \lp \frac{n-2j}{2r} \rp}  {\Gamma \lp \frac{n-2j}{2} \rp} \cdot \frac{1}{r} \cdot a_j(x,x),  &&
A_{j}(x)=\frac{(-1)^j}{j!} \cdot q_{rj}(x,x), \\
{}&A_{\frac{2j+1}{2r}}(x)=\frac{\Gamma \lp - \frac{2j+1}{2r} \rp}  {\Gamma \lp - \frac{2j+1}{2} \rp} \cdot \frac{1}{r} \cdot a_{\frac{n+2j+1}{2}}(x,x), &&
B_{l\frac{\beta}{2}}(x)=\frac{(-1)^{l\frac{\beta}{2}}}{r \lp l\frac{\beta}{2}  \rp ! \Gamma \lp - l\frac{\beta}{2} \cdot r \rp} \cdot a_{\frac{n+l \alpha}{2}}(x,x), 
\end{align*}
\[ A_{l \frac{\beta}{2}}(x)=   \frac{(-1)^{l \frac{\beta}{2}}}{(l \frac{\beta}{2})! \Gamma(-rl \frac{\beta}{2})} \cdot \FP_{s=-l \frac{\beta}{2}} \lp \Gamma(rs)q_{-rs}(x,x) \rp  +  \FP_{s=-l \frac{\beta}{2}} \lp \frac{\Gamma(s)}{\Gamma(rs)} \rp \cdot \frac{a_{\frac{n+l\alpha}{2}(x,x)}}{r} .  \]

\end{theorem}
\begin{proof}
We compute the coefficients from \eqref{ht2} by using proposition \ref{poligammaq}. 
\subsection*{The case when $n$ is even}
For $j \in \{0,1,...,n/2-1 \}$, we have: 
\begin{equation}\label{n-2j}
\Res_{s=\frac{n-2j}{2r}} \lp t^{-s} \frac{\Gamma(s)}{\Gamma(rs)} \Gamma(rs)q_{-rs}(x,x) \rp  =t^{-\frac{n-2j}{2r}} \cdot \frac{\Gamma(\frac{n-2j}{2r})}{\Gamma(\frac{n-2j}{2})} \cdot \frac{a_j(x,x)}{r}.
\end{equation}
The residue in $s=0$ is given by:
\begin{align*}
\Res_{s=0} \lp t^{-s} \Gamma(s) q_{-rs}(x,x) \rp ={}& \Res_{s=0} \lp t^{-s} \frac{\Gamma(s)}{\Gamma(rs)} \Gamma(rs)q_{-rs}(x,x) \rp \\
={}& r \cdot \frac{1}{r} \lp a_{\frac{n}{2}}(x,x) - \PP_{\Ker \DD}(x,x) \rp=a_{\frac{n}{2}}(x,x)-\PP_{\Ker \DD}(x,x),
\end{align*}
thus the coefficient of $t^0$ in the asymptotic expansion \eqref{ht2} is $a_{\frac{n}{2}}(x,x)$. 

\subsubsection*{The case when $n$ is even and $r$ is irrational} 
Let $j \in \mathbb N^*$. Then: 
\begin{equation}\label{-j}
\Res_{s=-j} \lp t^{-s} \Gamma(s) q_{-rs}(x,x) \rp = t^j \frac{(-1)^j}{j!} \cdot q_{rj}(x,x).
\end{equation}
Therefore in this case, the asymptotic expansion of $h_t$ is the following:
\[h_t(x,x) \stackrel{t \searrow 0}{\sim} \sum_{j=0}^{n/2-1} t^{-\frac{n-2j}{2r}} \frac{\Gamma \lp \frac{n-2j}{2r} \rp}{\Gamma \lp \frac{n-2j}{2} \rp} \frac{a_j(x,x)}{r} + a_{\frac{n}{2}}(x,x) + \sum_{j=1}^{\infty} t^j \frac{(-1)^j}{j!} q_{rj}(x,x).  \]

\subsubsection*{The case when $n$ is even and $r=\frac{\alpha}{\beta}$ is rational with $(\alpha,\beta)=1$} Some of the coefficients $q_{rj}(x,x)$ from \eqref{-j} can be expressed in terms of the $a_{k}$'s from \eqref{asp_t}.  Remark that $\frac{\Gamma(s)}{\Gamma(rs)}$ has simple poles in $\{-1,-2,... \} \setminus \{ \frac{-1}{r}, \frac{-2}{r},... \}$. For $j \in \mathbb N^*$, $s:=-\frac{j}{r} \in \{-1,-2,...\}$ if and only if $j$ is a multiple of $\alpha$, which is equivalent to $s=\frac{-l\alpha}{r}=-l\beta$ for some $l \in \mathbb N^*$. In this case, we obtain:
\begin{align*}
\Res_{s=-l\beta} \lp t^{-s} \Gamma(s) q_{-rs}(x,x) \rp ={}&
\Res_{s=-l\beta} \lp t^{-s} \frac{\Gamma(s)}{\Gamma(rs)}  \Gamma(rs)q_{-rs}(x,x) \rp \\
={}& t^{l\beta} r \cdot  \frac{1}{r} a_{\frac{n}{2}+l\alpha}(x,x)=t^{l\beta} a_{\frac{n}{2}+l \alpha}(x,x).
\end{align*}
Hence for rational $r=\frac{\alpha}{\beta}$, if $j=l\beta$, $l \in \mathbb N^*$, we conclude that:
\begin{equation}\label{npar}
q_{rj}(x,x)=(-1)^j \cdot j! \cdot a_{\frac{n}{2}+l\alpha}(x,x), 
\end{equation}
and $h_t(x,x)$ has the following asymptotic expansion as $t \searrow 0$:
\begin{align*}
 \sum_{j=0}^{n/2-1} t^{-\frac{n-2j}{2r}} \frac{\Gamma \lp \frac{n-2j}{2r} \rp}{\Gamma \lp \frac{n-2j}{2} \rp} \frac{a_j(x,x)}{r} + a_{\frac{n}{2}}(x,x) + \sum^{\infty}_{\substack{j=1  \\ 
\mathclap{\beta \nmid j}}}\mkern-6mu t^j \frac{(-1)^j}{j!} q_{rj}(x,x) + \sum_{l=1}^{\infty}t^{l\beta} a_{\frac{n}{2}+l\alpha}(x,x).
\end{align*}

\subsection*{The case when $n$ is odd}
For $j \in \lbrace 0,1,...,(n-1)/2 \rbrace$, the coefficient of $t^{-\frac{n-2j}{2r}}$ is computed as in \eqref{n-2j}. Furthermore, in $s=0$:
\begin{align*}
\Res_{s=0} \lp t^{-s} \Gamma(s) q_{-rs}(x,x) \rp ={}& \Res_{s=0} \lp t^{-s} \frac{\Gamma(s)}{\Gamma(rs)} \cdot \Gamma(rs)q_{-rs}(x,x) \rp \\
={}& r \cdot \frac{-1}{r} \cdot \PP_{\Ker \DD}(x,x) =-\PP_{\Ker \DD}(x,x),
\end{align*}
hence there is no free term in the asymptotic expansion of $h_t$ as $t$ goes to zero. 

Now we have to compute the residues of the function $t^{-s} \Gamma(s) q_{-rs}(x,x)$ in $s \in \{ -1,-2,... \}$ and $ s \in \{ \frac{-1}{2r}, \frac{-3}{2r},...\}$. 
\subsubsection*{The case when $n$ is odd and $r$ is irrational} Then these sets are disjoint, thus all poles of the function $\Gamma(s) q_{-rs}(x)$ are simple.  For $j \in \mathbb N^*$, the coefficient of $t^j$ is obtained as in \eqref{-j}. Furthermore, for $j \in \mathbb N$, we get:
\begin{equation}\label{2j+1/2r}
\Res_{s=-\frac{2j+1}{2r}} \lp t^{-s} \frac{\Gamma(s)}{\Gamma(rs)} \cdot \Gamma(rs)q_{-rs}(x,x) \rp 
 =t^{\frac{2j+1}{2r}} \cdot \frac{\Gamma(-\frac{2j+1}{2r})}{\Gamma(-\frac{2j+1}{r})} \cdot \frac{a_{\frac{n+2j+1}{2}}(x,x)}{r}.
\end{equation}
Therefore the small-time asymptotic expansion of $h_t$ is the following:
\begin{align*}
h_t(x,x) \stackrel{t \searrow 0}{\sim}{}& \sum_{j=0}^{n/2-1} t^{-\frac{n-2j}{2r}} \cdot \frac{\Gamma \lp \frac{n-2j}{2r} \rp}{\Gamma \lp \frac{n-2j}{2} \rp} \cdot \frac{a_j(x,x)}{r} + \sum_{j=1}^{\infty} t^j \cdot \frac{(-1)^j}{j!} q_{rj}(x,x)  \\ +{}& \sum_{j=0}^{\infty} t^{\frac{2j+1}{2r}} \cdot \frac{\Gamma \lp -\frac{2j+1}{2r} \rp}{\Gamma \lp - \frac{2j+1}{2} \rp} \cdot \frac{a_{\frac{n+2j+1}{2}}(x,x)}{r}.
\end{align*}
\subsubsection*{The case when $n$ is odd and $r=\frac{\alpha}{\beta}$ is rational} Consider the sets:
\begin{align*}
A:=\{ -1,-2,... \}, &&
B:=\{ \tfrac{-1}{2r},\tfrac{-3}{2r},... \}, &&
C:=\{ \tfrac{-1}{r},\tfrac{-2}{r},... \}.   
\end{align*}
Remark that $A$ is the set of negative poles of $s \longmapsto t^{-s}\Gamma(s)q_{-rs}(x,x)$, and $A \setminus C$ is the set of poles of the function $s \longmapsto \frac{\Gamma(s)}{\Gamma(rs)}$. Clearly $B$ and $C$ are disjoint. Moreover, $A \cap C=\{ -l\beta: \ l \in \mathbb N^* \}$. Furthermore, if $\beta$ is odd, then $A\cap B=\emptyset$, and otherwise if $\beta$ is even, then $A\cap B=\{-l\frac{\beta}{2}: \ l \in 2 \mathbb N+1 \}$. Such an $s=-\frac{2j+1}{2r}=l\frac{\beta}{2} \in A \cap B$ is a double pole for $\Gamma(s)q_{rs}(x)$.  
\subsubsection*{Suppose that $\beta$ is odd} Then $A$ and $B$ are disjoint. Thus for $s=-\frac{2j+1}{2r} \in B$, $j \in \mathbb N$, the residue of $t^{-s}\Gamma(s)q_{rs}(x,x)$ is the one computed in \eqref{2j+1/2r}. 

For $s=-j \in A\setminus C$ (which means that $j \in \mathbb N^*$, $\beta \nmid j$), the residue of $t^{-s}\Gamma(s)q_{-rs}(x,x)$ in $s$ is the one computed in \eqref{-j}.

If $s=-l\beta=-\frac{l\alpha}{r} \in A\cap C$ for some $l \in \mathbb N^*$, then $\Gamma(s)$ has a simple pole in $s$ and by proposition \ref{poligammaq}, $q_{-rs}(x,x)$ vanishes at $s=-l\beta$. Hence the product $t^{-s}\Gamma(s)q_{-rs}(x,x)$ is holomorphic in $s=-l\beta$ and $t^{l\beta}$, $l \in \mathbb N^*$, does not appear in the asymptotic expansion.

Therefore if $r=\frac{\alpha}{\beta}$ is rational and $\beta$ is odd, we obtain:
\begin{align*}
h_t(x,x) \stackrel{t \searrow 0}{\sim}{}& \sum_{j=0}^{n/2-1} t^{-\frac{n-2j}{2r}} \cdot \frac{\Gamma \lp \frac{n-2j}{2r} \rp}{\Gamma \lp \frac{n-2j}{2} \rp} \cdot \frac{a_j(x,x)}{r} + \sum_{j=0}^{\infty} t^{\frac{2j+1}{2r}} \cdot \frac{\Gamma \lp -\frac{2j+1}{2r} \rp}{\Gamma \lp - \frac{2j+1}{2} \rp} \cdot \frac{a_{\frac{n+2j+1}{2}}(x,x)}{r} \\ +{}& \sum^{\infty}_{\substack{j=1  \\ 
\mathclap{\beta \nmid j}}}\mkern-6mu t^j \frac{(-1)^j}{j!} \cdot q_{rj}(x,x) .
\end{align*}

\subsubsection*{Assume now that $\beta$ is even} For $s=-\frac{2j+1}{2r}\in B \setminus A$ ($j \in \mathbb N $ with $\alpha\nmid 2j+1$), the residue is computed as in \eqref{2j+1/2r}. For $s=-j \in A\setminus \lp B\cup C \rp$ (namely $j \in \mathbb N^*$, $\frac{\beta}{2} \nmid j$), the residue is computed as in \eqref{-j}.

For $s \in C \cap A$ (namely $s=-l\beta$, $l \in \mathbb N^*$), the residue is again $0$. Indeed, $\Gamma(s)$ has a simple pole in $-l\beta$ and by proposition \ref{poligammaq}, $q_{-rs}(x,x)$ vanishes in $-l\beta$, thus $t^{l\beta}$ does not appear in the asymptotic expansion of $h_t$.

Finally, if $s=-\frac{l\alpha}{2r}=-l\frac{\beta}{2} \in A \cap B$, $l \in 2 \mathbb N +1$, then $s$ is a double pole for $\Gamma(s)q_{-rs}(x,x)$. We write the Laurent expansions of the functions $t^{-s}$, $\frac{\Gamma(s)}{\Gamma(rs)}$ and $\Gamma(rs)q_{-rs}(x,x)$ respectively   in $s=-\frac{l\alpha}{2r}=-l\frac{\beta}{2}=:-k$:
\begin{align*}
{}&t^{-s}=t^k-t^k \log t + \mathcal O (s+k)^2, \\
{}&\frac{\Gamma(s)}{\Gamma(rs)}=\frac{(-1)^k}{k! \cdot \Gamma(-kr)} (s+k)^{-1}+..., \\
{}&\Gamma(rs)(q_{-rs}(x,x))=\frac{1}{r}a_{\frac{n+l\alpha}{2}}(x,x) (s+k)^{-1}+... .
\end{align*} 
Thus we finally obtain that:
\begin{align*}
\Res_{s=-k}\lp t^{-s} \cdot \frac{\Gamma(s)}{\Gamma(rs)} \cdot \Gamma(rs)q_{-rs}(x,x)\rp={}& t^k \cdot \frac{(-1)^k}{k! \Gamma(-kr)} \cdot \FP_{s=-k} \lp \Gamma(rs)q_{-rs}(x,x) \rp  \\
{}&+ t^k \FP_{s=-k} \lp\frac{\Gamma(s)}{\Gamma(rs)} \rp \cdot\frac{a_{\frac{n+l\alpha}{2}(x,x)}}{r} \\
{}&+t^k \log t \cdot \frac{(-1)^k}{k! \Gamma(-kr)} \frac{a_{\frac{n+l\alpha}{2}(x,x)}}{r}.  \qedhere
\end{align*}
\end{proof}

\section{Non-triviality of the coefficients}
Let us prove theorem \ref{nontriv}. Recall the definition of  the zeta function of a non-negative self-adjoint generalized Laplacian $\Delta$:
\[ \zeta_{\Delta}(s):=\sum_{\lambda \in \Spec \Delta \setminus \{ 0 \}} \lambda^{-s}=\int_{M} q_{-s}(x,x)dg(x). \]
This series is absolutely convergent for $\Re s > \frac{n}{2}$ and extends meromorphically to $\mathbb C $ with possible simple poles in the set: 
\[ \left\{ \frac{n}{2}-j : j \in \mathbb N \setminus \left\{ \frac{n}{2} \right\}  \right\}\]
(see for instance \cite{gilkey}).

Consider the trivial bundle $\mathbb C$ over a compact Riemannian manifold $M$. As in \cite{loya}, let $\lp \DD + \xi \rp_{\xi > 0}$ be a family of generalized Laplacians indexed by $\xi >0$, and denote by $ q_{-s}^{\xi} $ the Schwartz kernels of the operators $(\DD+\xi)^{-s}$. Note that for $\Re s > \frac{n}{2}$
\begin{equation}\label{zeta}
\int_M q_{-s}^{\xi} (x,x) dx = \Tr \lp \DD +\xi \rp^{-s}=\zeta_{\DD+\xi}(s) = \sum_{\lambda_j \in \Spec \DD} \lp \lambda_j +\xi \rp^{-s}. 
\end{equation}
Since for $\Re s > \frac{n}{2}$ the sum is absolutely convergent, we obtain: 
\[ \frac{d}{d\xi} \zeta_{\DD+\xi} (s) =-s \cdot \sum_{\lambda_j \in \Spec \DD} \lp \lambda_j +\xi \rp^{-s-1}= -s\cdot \zeta_{ \DD +\xi}( s+1). \]
By induction, it follows that for $\Re s > \frac{n}{2}$:
\begin{equation}\label{mer}
\frac{d}{d\xi^k} \zeta_{\DD+\xi}(s)=(-1)^k s(s+1)...(s+k-1) \cdot \zeta_{\DD +\xi}( s+k)
\end{equation} 
Using the identity theorem, \eqref{mer} holds true on $\mathbb C$ as an equality of meromorphic functions. Consider $s \in \mathbb R \setminus (-\mathbb N)$ and $k \in \mathbb N$ large enough such that $s+k > \frac{n}{2}$. Since $\zeta_{\DD+\xi}(s+k)$ is a convergent sum of strictly positive numbers, the right hand side is non-zero. Thus for any fixed $s \in \mathbb R \setminus (- \mathbb N) $, on any open set $U \subset \mathbb (0, \infty)$, the function $\xi \longmapsto \zeta_{\DD+\xi}(s)$ is not identically zero on $U$, and by \eqref{zeta}, $q_{-s}^{\xi}(x,x)$ cannot be constant zero on $M$. Hence for $s=-rj \notin - \mathbb N$, there exist $\xi_0 \in (0, \infty)$ and $x_0 \in M$ such that the coefficient $q_{rj}^{\xi_0}(x_0,x_0)$ of the asymptotic expansion of the Schwartz kernel $h_t$ of $e^{-t (\DD+\xi_0)^r}$ is non-zero.  

Now suppose that $rj \in \mathbb N$. Then $r=\frac{\alpha}{\beta}$ is rational and $j$ is a multiple of $\beta$, $j:=l\beta$. If $n$ is odd, we already proved in theorem \ref{asdiag} that $t^{l\beta}$ does not appear in the asymptotic expansion of $h_t$ as $t \searrow 0$. Furthermore, if $n$ is even, by \eqref{npar}, $q_{rj}(x,x)$ is a non-zero multiple of the coefficient $a_{\frac{n}{2}+l\alpha}(x,x)$ in the asymptotic expansion \eqref{asp_t} of the heat kernel $p_t$. It is well-known that the heat coefficients in \eqref{asp_t} are non-trivial, see for instance \cite{gilkey}. It follows that all coefficients obtained in theorem  \ref{asdiag} indeed appear in the asymptotic expansion, proving theorem \ref{nontriv}. 

\section{Non-locality of the coefficients $q_{rj(x)}$ in the asymptotic expansions}\label{nonlocal}
Let us prove theorem \ref{nonloc}. We give an example of an $n$-dimensional manifold and a Laplacian for which the coefficients $q_{rj}(x,x)$, $j=\overline{1,\infty}$, $rj \notin \mathbb N$ appearing in theorem \ref{asdiag} are not locally computable in the sense of definition \ref{def1} $i)$. Let $M=\mathbb R^n / \lp 2\pi \mathbb Z \rp^n$ be the $n$-dimensional product Riemannian manifold from example \ref{example}. Let $\DD_g$ be the Laplacian on $M$ given by the metric $g=d\theta_1^2 +...+ d\theta_n^2$. 

Remark that the eigenvalues of $\DD_g$ are $\{ k_1^2 +...+k_n^2 : k_1,...,k_n \in \mathbb Z   \}$.  Let $ \varphi_l(t)= \frac{1}{\sqrt{2\pi}} e^{il t} $ be the standard orthonormal basis of eigenfunctions of each $\DD_{S^1}$. Then for $\Re s > \frac{n}{2}$ and $\theta=(\theta_1,...,\theta_n) \in M$, the Schwartz kernel of $\DD_g^{-s}$ is given by:
\[ q_{-s}^{\DD_g} \lp \theta , \theta  \rp = \sum_{ (k_1,...,k_n) \in \mathbb Z^n \setminus  \{ 0 \} } \lp k_1^2+...+k_n^2 \rp^{-s} \varphi_{k_1}(\theta_1) \overline{\varphi_{k_1}(\theta_1)}... \varphi_{k_n}(\theta_n) \overline{\varphi_{k_n}(\theta_n)} . \]
Consider the $n$-dimensional zeta function
 \[ \zeta_n(s):=  \sum_{(k_1,...,k_n) \in \mathbb Z^n \setminus  \{ 0 \}} \lp k_1^2 +...+k_n^2 \rp^{-s}= \sum_{k \in \mathbb N^*} k^{-s} R_n(k), \]
where $R_n(k)$ is the number of representations of $k$ as a sum of $n$ squares.
Since $\varphi_{l}(t) \overline{\varphi_{l}(t)} =\frac{1}{2\pi}$ for any $t \in S^1$, it follows that:
\begin{equation}\label{m1}
  q_{-s}^{\DD_g} \lp \theta,\theta  \rp = \frac{1}{(2 \pi)^n} \zeta_n(s), 
\end{equation}
for any $\Re s > \frac{n}{2}$, and clearly $q_{-s}^{\DD_g}$ is independent of $\theta$.

Now let us change the metric locally on each component $S^1$. Let $U$ be an open interval in $S^1$, and $\psi:S^1 \longrightarrow [0,\infty)$ a smooth function with $\supp \psi \subset U$. Consider the new metric $\lp 1+\psi(\theta) \rp d\theta^2$ on each $S^1$. Then there exist $p >0$ and an isometry $\Phi: \lp S^1,  \lp 1+\psi(\theta) \rp d\theta^2 \rp \longrightarrow \lp S^1,  p^2 d\theta^2 \rp$. Remark that the Laplacian on $S^1$ given by the metric $p^2 d \theta^2$ corresponds under this isometry to $p^{-2}$ times the Laplacian for the metric $d\theta^2$. Let $\tilde{g}= \sum_{j=1}^n \lp 1+ \psi(\theta_j) \rp d\theta_j^2$ and $g_p=\sum_{j=1}^n p^2 d\theta_j^2$. Then clearly $\Phi \times...\times \Phi : (M, \tilde{g}) \longmapsto (M,  g_p) $ is an isometry, and let $\tilde{\Delta}$, $\Delta_{p }$ be the corresponding Laplacians on $M$. Denote by $q_{-s}^{\tilde{\Delta}}$ and $q_{-s}^{\Delta_{p }}$ the Schwartz kernels of the complex powers $ \tilde{\Delta} ^{-s}$ and $\Delta_{p }^{-s}$. We have for $\Re s >\frac{n}{2}$
\begin{equation}\label{ma}
 q_{-s}^{\DD_{p }} \lp \theta, \theta \rp = \frac{1}{(2 \pi p)^n} \sum_{k=(k_1,...,k_n) \in \mathbb Z^n \setminus  \{ 0 \}} \lp p^{-2}k_1^2 +...+p^{-2}k_n^2 \rp^{-s} = \frac{p^{2s}}{(2 \pi p)^n} \zeta_n(s) .
\end{equation}  
Remark that
\[  q_{-s}^{\DD_{p}} \lp \theta,\theta  \rp =q_{-s}^{\tilde{\DD}} \lp \Phi(\theta),\Phi(\theta) \rp,   \]
and both of them are independent of $\theta$. By \eqref{ma}, for $\Re s>\frac{n}{2}$, we obtain
\begin{equation}\label{m2}
q_{-s}^{\tilde{\DD}} \lp \theta,\theta \rp =  \frac{p^{2s-n}}{(2 \pi )^n}  \zeta_n(s) .
\end{equation}
Now we prove that $\zeta_n(s)$ has a meromorphic extension on $\mathbb C$ with so-called trivial zeros at $s=-1,-2,...$. By proposition \ref{mellin}, for $\Re s > \frac{n}{2}$ we have
\[  \zeta_n(s) \Gamma(s)=\int_{0}^{\infty} t^{s-1} \sum_{k=(k_1,...,k_n) \in \mathbb Z^n \setminus \{ 0 \}} e^{-t(k_1^2+...+k_n^2)}  dt = \int_0^{\infty} t^{s-1} F(t) dt,  \]
where $F(t):= \sum_{k=(k_1,...,k_n) \in \mathbb Z^n \setminus \{ 0 \}} e^{-t(k_1^2+...+k_n^2)}$. Using the multidimensional Poisson formula (see for instance \cite{bal}), it follows that:
\[ 1+F(t)= \sum_{k \in \mathbb Z^n} f_t(k)=\sum_{k \in \mathbb Z^n} \hat{f_t}(2 \pi k)=\pi^{n/2} t^{-n/2}  \lp 1 + F \lp \frac{\pi^2}{t} \rp \rp,   \]
and therefore
\[  F(t)=-1+\pi^{n/2}t^{-n/2} + \pi^{n/2} t^{-n/2} F \lp  \frac{\pi^2}{t}  \rp.   \]
Since $F(t)$ goes to $0$ rapidly as $t \to \infty$, the function $A(s)= \int_{1}^{\infty} t^{s-1} F(\pi t)dt  $
is entire. Remark that
\begin{align*}
 \zeta_n(s) \Gamma(s)={}&  \int_{0}^{\pi} t^{-s} F(t)dt + \int_{\pi}^{\infty} t^{s-1}F(t) dt  \\
={}& \pi^s \lp -\frac{1}{s} + \frac{1}{s-\frac{n}{2}} + A\lp \frac{n}{2}-s \rp + A(s)\rp,
\end{align*} 
so 
\begin{equation}
\pi^{-s} \zeta_n(s) \Gamma(s)= -\frac{1}{s} + \frac{1}{s-\frac{n}{2}} + A\lp \frac{n}{2}-s \rp + A(s).
\end{equation}
Therefore $\zeta_n$ extends meromorphically to $\mathbb C$ with a simple pole in $s=\frac{n}{2}$ and zeros at $s=-1,-2,...$. Furthermore, since the RHS is invariant through the involution $s \mapsto \frac{n}{2}-s$, it follows that $\zeta_n(s)$ does not have any other zeros for $s \in (-\infty,0)$. We obtain the well-known functional equation of the Epstein-zeta function:
\[ \pi^{-s} \zeta_n(s) \Gamma(s)= \pi^{s-n/2} \zeta_n \lp \frac{n}{2}-s  \rp \Gamma \lp \frac{n}{2}-s \rp,  \]
see for instance \cite[Eq.\ (63)]{chan}.

For $r \in (0,1)$ and $j=\overline{1,\infty}$ with $rj \notin \mathbb N$, $\zeta_n (-rj )$ is not zero, hence by \eqref{m1} and \eqref{m2} we get:
\[ q_{rj}^{\DD_g}(\theta, \theta) \neq q_{rj}^{\tilde{\DD}}(\theta,\theta).  \]
Since we modified the metric locally in $U^n \subset M$ and the corresponding kernel $q_{rj}^{\tilde{\DD}}$ changed its behavior globally, it follows that it is not locally computable in the sense of definition \ref{def1} $i)$. 

Furthermore, let us see that the coefficients $q_{rj}$ for $j=\overline{1,\infty}$, $rj \notin \mathbb N$ are not cohomologically local in the sense of definition \ref{def1} $iii)$. We argue by contradiction. Let $j$ be fixed. Suppose that there exists a function $C$ locally computable in the sense of definition \ref{def1} $i)$ such that:
\begin{align*}
 \int_{S^1\times...\times S^1} q_{rj}^{\DD_g} \dvol_g = \int_{S^1 \times...\times S^1}   C(g) \dvol_g,  && \int_{S^1\times...\times S^1} q_{rj}^{\tilde{\DD}} \dvol_{\tilde{g}} = \int_{S^1 \times...\times S^1}   C(\tilde{g}) \dvol_{\tilde{g}}. 
\end{align*}
Using \eqref{m1} and \eqref{m2}, it follows that
\begin{align*}
(2 \pi)^n \zeta_n(-rj) = \int_{S^1 \times...\times S^1}   C(g) \dvol_g , &&
&{}  (2 \pi p)^n p^{-2rj} \zeta_{n}(-rj)= \int_{S^1 \times...\times S^1}   C(\tilde{g}) \dvol_{\tilde{g}}. 
\end{align*}
Remark that $C(g)=C(\tilde{g})=C(g_p)$ must be constant functions, thus we have
\[ \int_{S^1 \times...\times S^1}   C \dvol_{\tilde{g}} = \int_{S^1 \times...\times S^1}   C \dvol_{g_p}   , \]
and since $(S^1 \times...\times S^1 , g_p)$ is locally isometric to $(S^1 \times...\times S^1 , g)$, we get $\int_{S^1 \times...\times S^1}   C \dvol_{g_p}  =  p^n \int_{S^1 \times...\times S^1}   C \dvol_g$. Therefore
\[ (2 \pi p)^n p^{-2rj} \zeta_n(-rj) = p^n \cdot (2 \pi)^n \zeta_n(-rj).    \]
Since $rj \notin \mathbb N$, we proved above that $\zeta_n (-rj)$ does not vanish. We obtain a contradiction because $p^{-2rj} \neq 1$ for $r \in (0,1)$, $j=1,2,...$.

\section{Interpretation of $h_t$ on the heat space for $r=1/2$}\label{unified}
In theorems \ref{asafaradiag} and \ref{asdiag}, we studied the asymptotic behavior of the heat kernel $h_t$ of ${ \DD^r}$, $r \in (0,1)$ for small time $t$ in two distinct cases: when we approach $t=0$ along the diagonal in $M \times M$, and when we approach a compact set away from the diagonal. We now give a simultaneous asymptotic expansion formula for both cases when $r=\frac{1}{2}$. Furthermore, in order to understand the asymptotic behavior as $t$ goes to zero in \emph{any} direction (not just the case when $t$ goes to $0$ in the vertical one),we will pull-back the formula on a certain \emph{linear} heat space $\HSp$. 

In \cite{melrose}, Melrose used his blow-up techniques to give a conceptual interpretation for the asymptotic of the heat kernel $p_t$. Recall that the heat space $M_H^2$ is obtained by performing a parabolic blow-up of $ \{ t=0 \} \times \Diag $ in $[0,\infty) \times M \times M$. The heat space $M_H^2$ is a manifold with corners with boundary hypersurfaces given by the boundary defining functions $\rho$ and $\omega_0$. The heat kernel $p_t$ belongs to $\rho^{-n} \mathcal C^{\infty}(M_H^2)$, and vanishes rapidly at the boundary hypersurface $\{ \omega_0=0 \}$ (see \cite[Theorem~7.12]{melrose}). 

In order to study the Schwartz kernel $h_t$ of $e^{-t \DD^{1/2}}$, we introduce the \emph{linear heat space} $\HSp$, which is just the standard blow-up of $\{ 0 \} \times \Diag$ in $[0,\infty) \times M \times M$ (see \cite{melmaz} for details regarding the blow-up of a submanifold). Let $\ff$ be the \emph{front face}, i.e. the newly added face, and denote by $\lb$ the \emph{lateral boundary} which is the lift of the old boundary $\{ 0 \} \times M \times M$. The blow down map is given locally by:
\begin{align*}
 \beta_H: \HSp \longrightarrow [0,\infty) \times M \times M && \beta_H(\rho, \omega, x')=(\rho \omega_0, \rho \omega'+x', x'),
\end{align*}
where
\[ \omega \in  \mathbb S^n_H= \{ \omega= (\omega_0, \omega') \in \mathbb R^{n+1} : \ \omega_0 \geq 0, \ \omega_0^2+|\omega'|^2 =1 \}.  \]

\begin{theorem}\label{unified1}
If $n$ is even,  then the Schwartz kernel $h_t$ of the operator $e^{-t\DD^{1/2}}$ belongs to  $ \rho^{-n}\omega_0 \cdot  \mathcal C^{\infty} (\HSp) $. If $n$ is odd, $h_t \in \rho^{-n} \omega_0 \cdot \mathcal C^{\infty} (\HSp) + \rho \log \rho \cdot \omega_0 \cdot \mathcal C^{\infty} (\HSp) $.
\end{theorem}
\begin{proof}
First we deduce the unified formula for $h_t$ as $t \searrow 0$ both on the diagonal and away from it.  By Mellin formula \ref{mellin} and inverse Mellin formula \ref{inversemellin}, for $\tau>n$, we get:
\begin{align*}
h_t(x,y)-\PP_{\Ker \DD}(x,y)={}& \tfrac{1}{2\pi i} \int_{\Re s=\tau} t^{-s} \frac{\Gamma(s)}{\Gamma \left( \frac{s}{2} \right)} \Gamma \left(\frac{s}{2} \right)q_{-s/2}(x,y)ds \\
={}& \tfrac{1}{2\pi i} \int_{\Re s = \tau} t^{-s} \frac{\Gamma(s)}{\Gamma \left( \frac{s}{2} \right)} \int_{0}^{\infty} T^{\tfrac{s}{2}-1} \lp p_T(x,y)-\PP_{\Ker \DD}(x,y)  \rp dT ds.
\end{align*}
We use the Legendre duplication formula as in \cite{art} (see for instance \cite{mellin}):
\[ \frac{\Gamma(s)}{\Gamma\lp \frac{s}{2} \rp } = \frac{1}{\sqrt{2\pi}} 2^{s-\tfrac{1}{2}} \Gamma \lp \frac{s+1}{2} \rp, \]
obtaining that $h_t(x,y)-\PP_{\Ker \DD}(x,y)$ is equal to:
\begin{align*}
\tfrac{1}{\sqrt{4 \pi}} \tfrac{1}{2\pi i} \int_{\Re s = \tau} \int_{0}^{\infty}  \lp \frac{2 \sqrt{T}}{t} \rp^s \Gamma \lp \frac{s+1}{2} \rp \lp p_T(x,y)-\PP_{\Ker \DD}(x,y)  \rp dT ds.
\end{align*}
Set $X:=\tfrac{2\sqrt{T}}{t}$. Using propositions \ref{umbenzi}, \ref{rapgamma} and Fubini, we first compute the integral in $s$. Changing the variable $S=\frac{s+1}{2}$ and applying the residue theorem, we get:
\begin{align*}
\tfrac{1}{2\pi i} \int_{\Re s = \tau} X^s \Gamma \lp \frac{s+1}{2} \rp ds={}& \tfrac{2}{2\pi i} \int_{\Re S = \frac{\tau+1}{2}} X^{2S-1} \Gamma(S) dS 
= 2 \sum_{k=0}^{\infty} \frac{(-1)^k}{k!} X^{-2k-1} \\
={}& 2 X^{-1}e^{-X^{-2}}=\frac{t}{\sqrt{T}}e^{-\frac{t^2}{4T}}.
\end{align*}
Thus we obtain:
\begin{equation}\label{dif}
\begin{aligned}
h_t(x,y)-\PP_{\Ker \DD}(x,y)=  \tfrac{t}{2\sqrt{\pi}} \int_{0}^{\infty} T^{-3/2} e^{-\frac{t^2}{4T}} \lp p_T(x,y)-\PP_{\Ker \DD}(x,y) \rp dT
\end{aligned}
\end{equation}
Since $p_T(x,y)-\PP_{\Ker \DD}(x,y)$ decays exponentially as $T$ goes to infinity, it follows that the right-hand side of equation \eqref{dif} is of the form $t \cdot \mathcal C^{\infty}_{t,x,y} \lp [0,\infty) \times M^2 \rp$. Furthermore, by the change of variable $u=\tfrac{t}{2\sqrt{T}}$, we have:
\begin{align*}
-\tfrac{t}{2\sqrt{\pi}}\int_{0}^{1} T^{-3/2}e^{-\frac{t^2}{4T}} dT \cdot \PP_{\Ker \DD}(x,y)
={}&-\tfrac{2}{\sqrt{\pi}}\int_{t/2}^{\infty} e^{-u^2} du \cdot \PP_{\Ker \DD}(x,y).
\end{align*}
Since $ \int_{t/2}^{\infty} e^{-u^2}du$ tends to $\frac{\sqrt{\pi}}{2}$ as $t \searrow 0$, the term $-\tfrac{t}{2\sqrt{\pi}}\int_{0}^{1} T^{-3/2}e^{-\frac{t^2}{4T}} dT  \PP_{\Ker \DD}(x,y)$ will cancel in the limit as $t \to 0$ with $- \PP_{\Ker \DD}(x,y)$ from the left hand side of \eqref{dif}. 

Let us study the remaining integral term $\tfrac{t}{2\sqrt{\pi}} \int_{0}^{1} T^{-3/2}e^{-\frac{t^2}{4T}} p_T(x,y)dT$. By theorem \ref{fi0}, 
\begin{align*}
p_T(x,y)=T^{-n/2} e^{-\frac{d(x,y)^2}{4T}} \sum_{j=0}^N T^j a_j(x,y) + R_{N+1}(T,x,y), 
\end{align*}
where the remainder $R_{N+1}(T,x,y)$ is of order $\mathcal O (T^{N+1})$, therefore:
\begin{align*}
\tfrac{t}{2\sqrt{\pi}} \int_{0}^{1} T^{-3/2}e^{-\frac{t^2}{4T}} p_T(x,y)dT={}&
\tfrac{t}{2\sqrt{\pi}} \int_{0}^{1} T^{-3/2}e^{-\frac{t^2}{4T}} R_{N+1}(T,x,y)  dT \\  +{}& \tfrac{t}{2\sqrt{\pi}} \int_{0}^{1} T^{-3/2}e^{-\frac{t^2}{4T}} T^{-n/2}e^{-\frac{d(x,y)^2}{4T}} \sum_{j=0}^{N} T^j a_{j}(x,y)   dT.
\end{align*}
Since $R_{N+1}(T,x,y)$ is of order $\mathcal O (T^{N+1})$, the first integral is again of type $ t \cdot \mathcal C^{\infty}_{t,x,y} $.  By changing the variable $u=\tfrac{t^2+d(x,y)^2}{4T}$ in the second integral, we get:
\begin{align*}
{}&\tfrac{t}{2\sqrt{\pi}} \sum_{j=0}^N a_{j}(x,y)  \int_{0}^{1}  T^{-\frac{n+3}{2} +j}e^{-\frac{t^2+d(x,y)^2}{4T}} dT  \\
={}&   \tfrac{t}{2 \sqrt{\pi}} \sum_{j=0}^N a_j(x,y)  \lp \frac{t^2+d(x,y)^2}{4} \rp^{-\frac{n+1}{2}+j}  \int_{\frac{t^2+d(x,y)^2}{4}}^{\infty} u^{\frac{n+1}{2}-j-1} e^{-u}du \\
={}& \tfrac{t}{2 \sqrt{\pi}} \sum_{j=0}^N a_j(x,y) \Gamma \lp \frac{n+1}{2}-j, \frac{t^2+d(x,y)^2}{4} \rp \lp \frac{t^2+d(x,y)^2}{4} \rp^{-\frac{n+1}{2}+j},  
\end{align*} 
where $\Gamma(z,\xi):=\int_{\xi}^{\infty} u^{z-1}e^{-u}du$ is the upper incomplete Gamma function.
We conclude that $h_t(x,y)$ is equal to:
\begin{equation}\label{htt}
 t \cdot \mathcal C^{\infty}_{t,x,y} + \tfrac{t}{2 \sqrt{\pi}} \sum_{j=0}^N a_j(x,y) \Gamma \lp \frac{n+1}{2}-j, \frac{t^2+d(x,y)^2}{4} \rp \lp \frac{t^2+d(x,y)^2}{4} \rp^{-\frac{n+1}{2}+j}.
\end{equation}

\subsection*{The case when $n$ is even} If $z>0$, then one can easily check that $\Gamma(z,\xi) \in \xi^z \mathcal C^{\infty}_{\xi}[0,\epsilon) + \Gamma(z)$, for some $\epsilon>0$. Furthermore, for $z \in (-\infty,0] \setminus \lbrace 0,-1,-2,...\rbrace$,  
\begin{align*}
\Gamma(z,\xi)={}&-\frac{1}{z}\xi^ze^{-\xi}+\frac{1}{z} \Gamma(z+1,\xi) \\
={}&\xi^z e^{-\xi} \sum_{k=0}^{a-1} \frac{-1}{z(z+1)...(z+k)} \xi^k + \frac{1}{z(z+1)...(z+a)} \Gamma(z+a,\xi) \\
={}& \xi^z \mathcal C^{\infty}_{\xi}[0, \epsilon) + \frac{1}{z(z+1)...(z+a-1)} \Gamma(z+a,\xi),
\end{align*}
where $a$ is a positive integer such that $z+a>0$. Thus for a non-integer $z<0$,  we have:
\begin{align*}
 \Gamma(z,\xi)= \xi^z \mathcal C^{\infty}_{\xi} [0, \epsilon) + \frac{1}{z(z+1)...(z+a-1)}  \Gamma(z+a).
\end{align*} 
We want to interpret equation \eqref{htt} on the heat space $\HSp$,  thus we pull-back \eqref{htt} through $\beta_H$:
\begin{align*}
 \beta_H^*h={}& \rho \omega_0 \beta_H^* \mathcal C^{\infty}_{t,x,y} + \tfrac{1}{2 \sqrt \pi} \rho \omega_0 \sum_{j=0}^{N} \lp \tfrac{\rho^2}{4} \rp^{-\frac{n+1}{2}+j}   \beta_H^*a_j(x,y) \Gamma\lp \frac{n+1}{2}-j , \frac{\rho^2}{4} \rp \\
 ={}& \rho \omega_0 \beta_H^* \mathcal C^{\infty}_{t,x,y} + \tfrac{1}{2 \sqrt{\pi}} \rho^{-n}\omega_0  \sum_{j=0}^{n/2} \rho^{2j} 2^{n+1-2j} \beta_H^*a_j(x,y) \Gamma \lp \frac{n+1}{2}-j \rp  \\
 {}&+ \tfrac{1}{2 \sqrt{\pi}} \rho \omega_0 \sum_{j=0}^{n/2} \beta_H^* a_j(x,y) \mathcal C^{\infty}_{\rho^2}[0,\epsilon) + \tfrac{1}{2 \sqrt{\pi}} \rho\omega_0 \sum_{j=n/2 +1}^{N} \beta_H^* a_j(x,y) \mathcal C^{\infty}_{\rho^2}[0,\epsilon) \\
 {}&+\tfrac{1}{2 \sqrt{\pi}} \rho^{-n}\omega_0 \sum_{j=n/2+1}^{N} \rho^{2j} 2^{n+1-2j} \beta_H^*a_j(x,y) \frac{2^{-n/2+j}}{\lp n+1-2j  \rp \lp n+3-2j  \rp...(-1)} \Gamma \lp \frac{1}{2} \rp. 
\end{align*} 
Since $\Gamma \lp \frac{n+1}{2} -j \rp =  \frac{\sqrt{\pi} (n-2j-1)!!}{2^{n/2-j}}  $ for $j \in \{0,1,...,n/2 \}$, it follows that:
\begin{equation}\label{rich}
\begin{aligned}
\beta_H^* h={}& \rho \omega_0 \beta_H^* \mathcal C^{\infty}_{t,x,y} + \omega_0 \rho \mathcal C^{\infty}_{\rho^2}[0,\epsilon) +  \rho^{-n} \omega_0 \sum_{j=0}^{n/2}  \rho^{2j} 2^{n/2-j} (n-2j-1)!! \beta_H^* a_j(x,y)  \\
{}&+ \rho^{-n}\omega_0 \sum_{j=n/2+1}^{N} \rho^{2j} \frac{ (-1)^{j-n/2} 2^{n/2-j} }{(2j-n-1)!!} \beta_H^*a_j(x,y).
\end{aligned}
\end{equation}
The case $\rho \neq 0$ and $\omega_0 \to 0$ corresponds to $x \neq y$ and $t \searrow 0$ before the pull-back. We obtain that $\beta_H^*h$ is in $\mathcal C^{\infty}(\HSp)$ and it vanishes at first order on $\lb$, which is compatible with theorem \ref{asafaradiag}.

If $\rho \to 0$ and $\omega_0=1$, which corresponds to $x=y$ and $t \searrow 0$, then $\beta_H^*h=  \rho^{-n}\omega_0 \sum_{j=0}^{N} \rho^{2j} A_{j}(x)$, where we denoted by $A_j(x)$ the coefficients appearing in \eqref{rich}. Again, this result is compatible with theorem \ref{asdiag}, and moreover, the coefficients are precisely the ones from \cite[Theorem~3.1]{art}. 

Remark that formula \eqref{rich} is stronger than theorems \ref{asafaradiag} and \ref{asdiag}. If both $\rho$ and $\omega_0$ tend to $0$ (with different speeds), it describes the behavior of $h_t$ as $t$ goes to zero from any positive direction (not only the vertical one). 
 
\subsection*{The case when $n$ is odd} Remark that for small $\xi$, we have: 
\begin{align*}
\Gamma(0,\xi)={}&\int_{\xi}^{\infty} t^{-1}e^{-t} dt = \int_{\xi}^{1} \frac{e^{-t}-1}{t}  dt +\int_{\xi}^{1} t^{-1} dt + \int_{1}^{\infty}t^{-1}e^{-t}dt \\
={}&- \log \xi + \mathcal C^{\infty}_{\xi}[0,\epsilon).
\end{align*} 
Furthermore,  if $p$ is a negative integer, inductively we obtain:
\begin{align*}
\Gamma(-p,\xi)={}& \frac{e^{-\xi}\xi^{-p}}{p!} \sum_{k=0}^{p-1} (-1)^k (p-k-1)! \xi^k + \frac{(-1)^p}{p!}  \Gamma(0,\xi) \\
={}& \xi^{-p} \mathcal C^{\infty}_{\xi}[0,\epsilon)-\frac{(-1)^p}{p!} \log \xi +\mathcal C^{\infty}_{\xi}[0,\epsilon).  
\end{align*} 
We pull-back equation \eqref{htt} on the heat space $\HSp$:
\begin{align*}
 \beta_H^*h={}& \rho \omega_0 \beta_H^* \mathcal C^{\infty}_{t,x,y} + \tfrac{1}{2 \sqrt \pi} \rho \omega_0 \sum_{j=0}^{N} \lp \tfrac{\rho^2}{4} \rp^{-\frac{n+1}{2}+j}   \beta_H^*a_j(x,y) \Gamma\lp \frac{n+1}{2}-j , \frac{\rho^2}{4} \rp \\
 ={}& \rho \omega_0 \beta_H^*a_j(x,y)  + \tfrac{1}{2\sqrt{\pi}} \rho \omega_0 \sum_{l=0}^{(n-1)/2} \beta_H^* a_j(x,y) \mathcal C^{\infty}_{\rho^2}[0,\epsilon) \\
 {}&+\frac{1}{\sqrt{\pi}} \rho^{-n}\omega_0 \sum_{j=0}^{(n-1)/2} \rho^{2j} \beta_H^*a_j(x,y) 2^{n-2j} \Gamma \lp \frac{n+1}{2}-j \rp \\
 {}&+ \frac{2}{\sqrt{\pi}} \rho^{-n} \omega_0 \sum_{j=(n+1)/2}^{N} \rho^{2j} \log \rho \beta_H^{*} a_j(x,y) 2^{n-2j} \frac{(-1)^{j-\frac{n+1}{2}+1}}{\lp j-\frac{n+1}{2} \rp!} \\
{}&+ \frac{2}{\sqrt{\pi}} \rho^{-n} \omega_0 \sum_{j=(n+1)/2}^{N} \rho^{2j}  \beta_H^{*} a_j(x,y) 2^{n-2j} \frac{(-1)^{j-\frac{n+1}{2}}}{\lp j-\frac{n+1}{2} \rp!} \log 2  \\
{}&+\tfrac{1}{2\sqrt{\pi}} \rho \omega_0 \sum_{j=(n+1)/2}^{N} \beta_H^*a_j(x.y) \mathcal C^{\infty}_{\rho^2}[0,\epsilon) \\
{}&+ \tfrac{1}{ \sqrt{\pi}} \rho^{-n}\omega_0 \sum_{j=(n+1)/2}^{N} \rho^{2j} \beta_H^*a_j(x,y) 2^{n-2j} \frac{(-1)^{j-\frac{n+1}{2}}}{\lp j-\frac{n+1}{2} \rp!} \mathcal C^{\infty}_{\rho^2}[0,\epsilon) 
\end{align*}
Therefore, we obtain:
\begin{equation}\label{rich'}
\begin{aligned}
\beta_H^* h={}& \rho \omega_0 \beta_H^* \mathcal C^{\infty}_{t,x,y} + \omega_0 \rho \mathcal C^{\infty}_{\rho^2}[0,\epsilon) + \omega_0 \rho^{-n} \mathcal C^{\infty}_{\rho^2}[0,\epsilon)  \\
{}&+ \frac{1}{\sqrt{\pi}} \rho^{-n} \omega_0 \sum_{j=0}^{(n-1)/2} \rho^{2j} \beta_H^*a_j(x,y) 2^{n-2j} \lp \frac{n+1}{2}-j \rp ! \\
{}&+  \frac{2}{\sqrt{\pi}} \rho^{-n} \omega_0 \sum_{j=(n+1)/2}^{N} \rho^{2j} \log \rho \beta_H^{*} a_j(x,y) 2^{n-2j} \frac{(-1)^{j-\frac{n+1}{2}+1}}{\lp j-\frac{n+1}{2} \rp!} \\
{}&+ \frac{2}{\sqrt{\pi}} \rho^{-n} \omega_0 \sum_{j=(n+1)/2}^{N} \rho^{2j}  \beta_H^{*} a_j(x,y) 2^{n-2j} \frac{(-1)^{j-\frac{n+1}{2}}}{\lp j-\frac{n+1}{2} \rp!} \log 2.
\end{aligned}
\end{equation}

If $\rho \neq 0$ and $\omega_0 \to 0$ (corresponding to $x \neq y$ and $t \searrow 0$ before the pull-back on $\HSp$), we obtain that $\beta_H^* h \in \mathcal C^{\infty}(M_{heat})$ and it vanishes at order $1$ at $\lb$, which is compatible with the result of theorem \ref{asafaradiag}.

In the case $\rho \to 0$ and $\omega_0=1$ which corresponds to $x=y$ and $t \searrow 0$, we obtain $\beta_H^* h= \rho^{-n}\mathcal C^{\infty}_{\rho^2} + \rho^{-n}\sum_{j=0}^N \rho^{2j} A_j(x) + \rho^{-n} \sum_{j=(n+1)/2}^N \rho^{2j} \log \rho B_j(x)$, where we denoted by $A_j$ and $B_j$ the coefficients appearing in \eqref{rich'}.  This result is compatible with theorem \ref{asdiag} and again, we find some of the coefficients appearing in \cite[Theorem~3.1]{art}. 
\end{proof}

\section{The heat kernel as a polyhomogeneous conormal section}
Let us recall the notions of index family and polyhomogeneous conormal functions on a manifold with corners with two boundary hypersurfaces. (For an accessible introduction see \cite{grieser}, and for full details of the theory see \cite{mel92}.) A discrete subset $F \in \mathbb C \times  \mathbb N $ is called an \emph{index set} if the following conditions are satisfied:
\begin{itemize}
\item[1)] For any $N \in \mathbb R$, the set $F \cap \{  (z,p): \Re z < N  \} $ is finite.
\item[2)] If $p > p_0$ and $(z,p) \in F$, then $(z, p_0) \in F$.
\end{itemize}
If $X$ is a manifold with corners with two boundary hypersurfaces $B_1$ and $B_2$ given by the boundary defining functions $x$ and $y$, a smooth function $f$ on $\accentset{\circ}{X}$ is said to be \emph{polyhomogeneous conormal} with index sets $E$ and $F$ respectively, if in a small neighborhood $[0,\epsilon) \times B_1$, $f$ has the asymptotic expansion:
\[  f(x,y) \stackrel{x \searrow 0}{\sim} \sum_{(z,p) \in F} a_{z,p}(y)\cdot x^z \log^p x,  \]
where $a_{z,p}$ are smooth coefficients on $B_2$, and for each $a_{z,p}$ there exists a sequence of real numbers $b_{w,q}$ such that
\[  a_{z,p}(y) \stackrel{y \searrow 0}{\sim} \sum_{(w,q) \in E}b_{w,q} \cdot y^w \log^q y.\]

One can prove that $f$ is a polyhomogeneous conormal function on $X$ with index sets $F_p= \{ (k,0) : k \in \mathbb Z, k \geq -p  \}$ and $F_0=\{ (n,0) : n \in \mathbb N \}$ if and only if $f \in  y^{-p} \mathcal C^\infty(X)$. Furthermore $f$ is a polyhomogeneous conormal function on $X$ with index sets $F'=\{ (n,1) : n \in \mathbb N^* \}$ and $F_0$ if and only if $f \in  \mathcal C^\infty(X)+ \log y \cdot \mathcal C^{\infty}(X)$.
Therefore we can restate theorem \ref{unified1} as follows
\begin{theorem}
For $r=\frac{1}{2}$, the heat kernel $h_t$ of the operator $e^{-t \DD^{1/2}}$ is a polyhomogeneous conormal section on the linear heat space $\HSp$ with values in $\EE \boxtimes \EE^*$.  The index set for the lateral boundary is:
\[ F_{\lb} =\{ (k,0): k \in \mathbb N^* \}.  \]
If $n$ is even,  the index set of the front face is:
\[  F_{\ff}=\{(-n+k,0): k \in \mathbb N \},  \]
while for $n$ odd the index set towards $\ff$ is given by:
\[  F_{\ff}=\{ (-n+k,0): k \in \mathbb N \} \cup \{ (k,1) : k \in \mathbb N^* \}.  \]
\end{theorem}

It seems reasonable to expect that the Schwartz kernel $h_t$ of the operator $e^{-t\DD^r}$ for $r \in (0,1)$ can be lifted to a polyhomogeneous conormal section in a certain ``transcendental" heat space $M^r_{Heat}$ depending on $r$ with values in $\EE \boxtimes \EE^*$. However, already in the case $r=1/3$ our method leads to complicated computations involving Bessel modified functions. We therefore leave this investigation open for a future project.


\begin{thebibliography}{99}
\bibitem{agrano}
~M.S.~Agronovi\v{c}, Some asymptotic formulas for elliptic pseudodifferential operators, \emph{Funktsional. Anal. i Prilozhen.} {\bf 21} (1987), 63-65.

\bibitem{art}
~C.~B\"{a}r, S.~Moroianu, Heat Kernel Asymptotics for Roots of Generalized Laplacians, \emph{Int. J. Math.} {\bf 14} (2003), 397-412.  

\bibitem{bal}
~R.~Bellman, \emph{A brief introduction to theta functions}, Athena Series: Selected Topics in Mathematics Holt, Rinehart and Winston, New York (1961).

\bibitem{berline}
~N.~Berline, E.~Getzler, M.~Vergne, \emph{Heat Kernels and Dirac Operators}, Springer Verlag, (2004). 

\bibitem{berver}
~N.~Berline, M.~Vergne, A computation of the equivariant index of the Dirac operator, \emph{Bull. Soc. Math France} {\bf 113} (1985), 305-345.

\bibitem{gauduchon}
~M.~Berger, P.~Gauduchon, E.~Mazet, \emph{Le spectre d'une vari\'et\'e riemannienne}, Lecture Notes in Mathematics {\bf 194}, Springer, Berlin New York (1971). 

\bibitem{bismut}
~J.M.~Bismut, The Atiyah-Singer theorems: a probabilistic approach, \emph{J. Funct. Anal. } {\bf 57} (1984), 329-348.

\bibitem{spinorial}
J.~Bourguignon, O.~Hijazi, J.~Milhorat, A.~Moroianu, S.~Moroianu, \emph{A Spinorial Approach to Riemannian and Conformal Geometry}, European Mathematical Society, (2015). 

\bibitem{chan}
~K.~Chandrasekharan,  R.~Narasimhan,  Hecke's functional equation and arithmetical identities, \emph{Ann. of Math. (2)} {\bf 74} (1961), 1–23.

\bibitem{duis}
~J.J.~Duistermaat, V.W.~Guillemin, The spectrum of positive elliptic operators and periodic bicharacteristics, \emph{Invent Math} {\bf 29} (1975), 39-79.  

\bibitem{fah}
~M.A.~Fahrenwaldt, Off-diagonal heat kernel asymptotics of pseudodifferential operators on closed manifolds and subordinate Brownian motion, \emph{Integral Equations Operator Theory} {\bf 87} (2017), 327-347.  

\bibitem{getzler}
~E.~Getzler, Pseudodifferential operators on supermanifolds and the index theorem, \emph{Comm. Math. Phys.} {\bf 92} (1983), 163-178.  

\bibitem{gilkey}
~P.B.~Gilkey, \emph{Invariance theory, the heat equation, and the Atiyah-Singer index theorem}, Second edition, Studies in Advanced Mathematics, CRC Press, Boca Raton, FL, (1995).

\bibitem{gilgrubb}
~P.B.~Gilkey, G.~Grubb, Logarithmic terms in asymptotic expansions of heat operator traces, \emph{Comm. Partial Differential Equations} {\bf 23} (1998), no. 5-6, 777-792. 

\bibitem{grieser}
~D.~Grieser, Basics of the b-calculus, In J.B. Gil, D. Grieser, and M. Lesch, editors, \emph{Approaches to Singular Analysis}, Advances in Partial Differential Equations, pages 30-84, Basel, 2001, Birkh\''ou ser.

\bibitem{gru}
~G.~Grubb, \emph{Functional calculus of pseudo-differential boundary problems}, Progress in Math. 65, Birkh\"{a}user, Boston, 1986.

\bibitem{loya}
~P.~Loya, S.~Moroianu, R.~Ponge, On the singularities of the zeta and eta functions of an elliptic operator, \emph{Internat. J. Math.} {\bf 23} (2012), no. $6$, $1250020$, $26$pp.  

\bibitem{pleijel}
~S.~Minakshisundaram, A.~Pleijel, Some properties of the eigenfunctions of the Laplace operator on Riemannian manifolds, \emph{Can. J. Math. } {\bf 1} (1949), 242-256.  

\bibitem{mel92}
~R.B.~Melrose, Calculus of conormal distributions on manifolds with corners, \emph{Internat. Math. Res. Notices} 1992, no. 3, 51–61. 

\bibitem{melrose}
~R.B.~Melrose, \emph{The Atiyah-Patodi-Singer Index Theorem}, Research Notes in Mathematics, 4. A K Peters, Ltd., Wellesley, MA, (1993). 

\bibitem{melmaz}
~R.B.~Melrose, R.R.~Mazzeo, Analytic surgery and the eta invariant, \emph{Geom. Funct. Anal.}  {\bf 5} (1995), no. $1$, 14-75.  

\bibitem{mellin}
~R.B.~Paris, D.~Kaminski, \emph{Asymptotics and Mellin-Barnes Integrals}, Cambridge University Press, (2001). 

\bibitem{watson}
~E.T.~Whittaker, G.N.~Watson, \emph{A Course of Modern Analysis}, Cambridge University Press, (1965). 

\end{thebibliography}
\end{document}